\documentclass[10pt]{amsart}

\usepackage{amsthm,amssymb,epsfig,graphicx,latexsym,mathdots}
\usepackage{graphicx,amsfonts,subfigure,dsfont,scalerel}
\usepackage{hyperref}
\usepackage[usenames, dvipsnames]{color}

\numberwithin{equation}{section}

\theoremstyle{plain}
\newtheorem{theorem}{Theorem}[section]
\newtheorem{lemma}[theorem]{Lemma}
\newtheorem{proposition}[theorem]{Proposition}
\newtheorem{corollary}[theorem]{Corollary}
\newtheorem{definition}[theorem]{Definition}

\theoremstyle{remark}
{
    
    \newtheorem{remark}[theorem]{Remark}

}

\newcommand{\norm}[2]{\left\lVert#1\right\rVert_{#2}}

\newcommand{\cU}{\mathcal{U}}

\newcommand{\A}{\hat{A}}
\newcommand{\B}{\hat{B}}

\newcommand{\hw}{\hat{w}}

\newcommand{\Y}{\mathcal Y}
\newcommand{\Uad}{\mathcal{U}_{ad}}

\newcommand{\xiz}{\xi}
\newcommand{\hatQ}{\widehat \Q}
\newcommand{\id}{{\rm id}}

\def\dd{{\rm d}}
\newcommand{\ddt}{\frac{\rm d}{{\rm d} t} }

\def\weight(#1,#2){c_{#1,#2}}

\def\ra{{\rm ra}}

\if {

}{{\caption}}
} \fi

\def\hh{\hat{h}}

\def\ph{\hat{p}}

\def\uh{\hat{u}}

\def\wh{\hat{w}}

\def\yh{\hat{y}}

\def\cala{{\mathcal  A}}
\def\calb{{\mathcal B}}

\def\cald{{\mathcal D}}

\def\calh{{\mathcal H}}

\def\call{{\mathcal L}}

\def\calt{{\mathcal T}}
\def\calu{{\mathcal U}}

\def\caly{{\mathcal Y}}

\def\cF{{\mathcal F}}

\def\cR{{\mathcal R}}

\def\A{\mathcal{A}}
\def\B{\mathcal{B}}

\def\H{\mathcal{H}}

\def\L{\mathcal{L}}

\def\Q{\mathcal{Q}}

\def\T{\mathcal{T}}
\def\U{\mathcal{U}}

\def\Y{\mathcal{Y}}

\def\eps{\varepsilon}
\def\Om{{\Omega}}
\def\om{{\omega}}

\def\gamh{\hat{\gamma}}

\def\psih{{\hat\psi}}
\def\Psih{{\hat\Psi}}
\def\Psih{{\hat\Psi}}

\def\xih{{\hat\xi}}

\def\Psih{{\hat\Psi}}

\def\Psih{{\hat\Psi}}

\def\1B{{\bf  1}}

\def\dist{\mathop{\rm dist}}
\def\ddiv{\mathop{\rm div}}
\def\dom{\mathop{{\rm dom}}}

\def\Min{\mathop{\rm Min}}

\def\dPsi{\mathop{\delta \Psi}}

\def\half{\mbox{$\frac{1}{2}$}}

\def\1B{{\bf  1}}

\newcommand{\NN}{\mathbb{N}}
\newcommand{\RR}{\mathbb{R}}

\def\cN{\mathbb{N}}

\def\cR{\mathbb{R}}

\newcommand\be{\begin{equation}}
\newcommand\ee{\end{equation}}
\newcommand\ba{\begin{array}}
\newcommand\ea{\end{array}}
\newcommand{\bea}{\begin{eqnarray}}
\newcommand{\eea}{\end{eqnarray}}
\newcommand{\bean}{\begin{eqnarray*}}
\newcommand{\eean}{\end{eqnarray*}}

\def\rar{\rightarrow}

\def\disp{\displaystyle}

\def\la{\langle}
\def\ra{\rangle}

\title[Optimal control of infinite dimensional bilinear systems]
{Optimal control of infinite dimensional \\ bilinear systems:
Application to\\ the heat and wave equations}

\author{M. Soledad Aronna}
\address{EMAp/FGV, Rio de Janeiro 22250-900, Brazil}
\email{soledad.aronna@fgv.br}

\author{Fr\'ed\'eric Bonnans}
\address{INRIA-Saclay and Centre de 
Math\'ematiques Appliqu\'ees, Ecole Polytechnique, 91128 Palaiseau, France}
\email{Frederic.Bonnans@inria.fr}  

\author{Axel Kr\"oner}
\address{INRIA-Saclay and Centre de 
Math\'ematiques Appliqu\'ees, Ecole Polytechnique, 91128 Palaiseau, France}
\email{Axel.Kroener@inria.fr}  

\thanks{This article will appear in {\em Mathematical Programming}.}

\begin{document}

\maketitle


\begin{abstract}
  In this paper we consider second order optimality conditions for a bilinear optimal control
  problem governed by a strongly continuous semigroup operator,
the control entering linearly in the cost function.
We derive first and second order optimality conditions, taking
advantage of the Goh transform. We then apply the results to the
 heat and wave equations. 
\end{abstract}

\vspace{5mm}


\section{Introduction}
In this paper we derive no gap second order optimality conditions for optimal control problems governed by a bilinear system  being affine-linear in the control and with pointwise constraints on the control; more precisely for a Banach space $\calh$ we consider optimal control problems for equations of type
\be
\label{semg1}
\dot\Psi + \cala \Psi = f + u (\B_1 + \B_2\Psi); 
\;\; t\in (0,T); \quad \Psi(0)=\Psi_0,
\ee
where  $\A$ is the 
generator of a strongly continuous semigroup on $ \calh$, and
\be
\label{semg1-hyp}
\Psi_0\in \calh;\;\;  f\in L^1(0,T;\calh);\;\; 
\B_1\in \calh;\;\; 
u\in L^1(0,T);\;\; 
\B_2\in \L(\calh).
\ee
This general framework includes in particular optimal control problems for the bilinear heat and wave equations.
 
Optimal control problems which are affine-linear in the control are important when addressing problems with $L^1$-control costs. 
However, for affine-linear control problems, the classical techniques of the calculus of variations do not lead to the formulation of second order sufficient optimality conditions. This problem has been studied in the context of optimal control of ordinary differential equations (ODEs) based on the Legendre condition by Kelly \cite{Kel64}, Goh \cite{Goh66a}, Dmitruk \cite{Dmi77,MR914861}, Poggiolini and Stefani \cite{MR2472886}, Aronna et al.
\cite{ABDL12}, and Frankowska and Tonon
\cite{frankowska:hal-01067270}; the case of additional state
constraints was considered in Aronna et al. \cite{AronnaBonnansGoh2016}. 
In the context of optimal control of PDEs there exist only a few
papers on sufficient optimality conditions for affine-linear control
problems, see Bergounioux and Tiba \cite{MR1377719}, Tr\"oltzsch
\cite{MR2144184}, Bonnans and Tiba \cite{MR2656166}, 
who discuss generalized bang-bang control.
Bonnans \cite{Bonnans:2013} discussed singular arcs
in the framework of semilinear parabolic equations.
Let us also mention the results on second order necessary or
sufficient
conditions by Casas \cite{MR2974742} (for the elliptc case),
Casas and Tr\"oltzsch (review paper \cite{MR3311948}),
Casas, Ryll and Tr\"oltzsch 
({F}itz{H}ugh-{N}agumo equation \cite{MR3374636}).

Further, for optimal control of semigroups, the reader is referred to
Li et al.~\cite{LiYao:1985,LiYong:1991}, Fattorini et al. 
\cite{FattoriniFrankowska:1991,MR2158806} 
and Goldberg and Tr\"oltzsch \cite{MR1227544}.

The contribution of this paper is to derive sufficient second order
optimality condition using the Goh transform \cite{Goh66a}. We
generalize ideas in \cite{Bonnans:2013} 
to the case of bilinear systems, in a semigroup setting. A general framework is
presented which allows to obtain sufficient optimality conditions
under very general hypotheses. We verify additionally that these conditions are
satisfied in the case of control of
the heat and wave equations. We also discuss the case of
a general diagonalizable operator.
In the companion paper \cite{aronna:hal-01311421},
we wil extend these results to the case of complex spaces, 
with an application to the Schr{\"o}dinger equation.

The paper is organized as follows. 
Section \ref{sec:abset} presents the abstract control problem in
a semigroup setting and establishes some basic calculus rules.
Necessary second order optimality conditions are presented in 
Section \ref{sec:soocg}.
Sufficient ones are the subject of Section \ref{sec-suffici}.
Applications to the control of the heat equation and wave  equation
are presented in Section \ref{sec:applschr}.

\vspace{5pt}

\noindent {\bf Notation.}
Given a Banach space $\H$, with norm 
$\|\cdot \|_\H$, we denote by $\H^*$ its
topological dual and by $\la h^*,h\ra_\H$ the duality
product between $h\in \H$ and $h^*\in \H^*$. 
We omit the index $\H$ if there is no ambiguity.
If $\A$ is a linear (possibly unbounded) operator 
from $\H$ into itself, its adjoint operator is denoted by $\A^*$.
We let $|\cdot|$ denote the Euclidean norm and $AC(0,T)$ the space of absolutely continuous functions over $[0,T]$. By $\|\cdot\|_p$, for 
$p \in [1,\infty]$, we mean by default the norm of $L^p(0,T)$. 

\section{The abstract control problem in
a semigroup setting}\label{sec:abset}

\subsection{Semigroup setting}\label{sec:absetsese}

Let $\calh$ be a reflexive Banach space.
Consider the abstract differential equation 
\eqref{semg1} 
with data satisfying \eqref{semg1-hyp},
the unbounded operator $\A$ over $\H$ 
being the generator of a 
(strongly) continuous semigroup denoted by $e^{-t\A}$,
such that 
\be\label{cond:bounded}
\| e^{-t \A}\|_{\L(\calh)} \leq c_\A e^{\lambda_\A t},\quad t>0,
\ee
for some positive $c_\A$ and $\lambda_\A$.
Thus (\cite[Ch. 1, Cor. 2.5]{MR710486})  
$\A$ is a closed operator and has dense domain defined by
\be
\label{def-domA}
\dom(\A) := \left\{ y\in \H; \;\; \lim_{t\downarrow 0} 
\frac{e^{-t\A}y-y}{t} \;\text{exists} \right\}
\ee
and, for $y\in \dom(\A)$: 
\be
\label{def-domAa}
\A y = - \lim_{t\downarrow 0} \frac{e^{-t\A}y-y}{t}.
\ee
We define the {\em mild solution} of
\eqref{semg1} as the function
$\Psi\in C(0,T;\calh)$ such that,
for all $t\in [0,T]$:
\be
\label{semPsi}
\Psi(t) = e^{-t\A}\Psi_0 + 
\int_0^t e^{-(t-s)\A} \big( f(s)+ u(s) (\B_1+ \B_2 \Psi(s)) \big) \dd s.
\ee 
This fixed-point equation has a unique solution in 
$C(0,T;\H)$.
Indeed, letting $\calt(\Psi)(t)$ denote
the r.h.s. of \eqref{semPsi}, we see that 
$\calt$ is a continuous mapping from
$C(0,T;\calh)$ into itself, and that given 
$\Psi^1,\Psi^2$ in that space
we have that 
\be
\calt(\Psi^1)(t) - \calt(\Psi^2)(t)
= 
\int_0^t e^{-(t-s)\A}  u(s) \B_2 \Big(\Psi^1(s)-\Psi^2(s)\Big) \dd s.
\ee
For $t$ small enough, this is a contracting operator
and, by induction, we deduce that this equation is
well-posed.
We let $\Psi[u]$ denote the unique solution of \eqref{semPsi} for each $u\in L^1(0,T).$

We recall that the adjoint of $\A$ is defined as
follows:
its domain is 
\be
\dom(\A^*) := \{ \varphi\in \H^*; \;\; \text{for some $c>0$:}\;\;
| \la \varphi, A y \ra | \leq c \|y\|, \;\; \text{for all $y\in \dom(\A)$}\},
\ee
so that $y\mapsto \la \varphi, A y \ra$
has a unique extension to a linear continuous form over $\H$,
which by the definition is $\A^*\varphi$. 
This allows to define weak solutions \cite{Ball:1977}:

\begin{definition}
\label{ThmAC-def}
We say that $\Psi\in C(0,T;\H)$ is a
{\em  weak solution} of \eqref{semg1}
if  $\Psi(0)= \Psi_0$ and, for any 
$\phi \in \dom(\A^*)$,  the function $t\mapsto \la \phi , \Psi(t)\ra$ 
is absolutely continuous over $[0,T]$ and satisfies
\begin{align}
\label{weak_ball}
\ddt\la \phi , \Psi(t)\ra+\la \A^*\phi , \Psi(t)\ra 
= \la \phi , f+u(t)(\B_1+\B_2 \Psi(t))\ra,
\;\; \text{for a.a. $t\in [0,T]$.}
\end{align}
\end{definition}

We recall the following result,  see \cite{Ball:1977}:

\begin{theorem}
\label{weaksense-thm}
Let $\A$ be the generator of a strongly continuous semigroup.
Then there is a unique weak solution of \eqref{weak_ball} that coincides with the
mild solution.
\end{theorem}

So in the sequel we can use any of the two equivalent formulations 
\eqref{semPsi} or \eqref{weak_ball}.

Let us set 
$\zeta (t) := w(t) y(t)$, 
where $w$ is a primitive of $v\in L^1(0,T)$
such that $w(0)=0$,  
and $y\in C(0,T;\H)$ is a mild solution
for some $b\in L^1(0,T;\calh)$:
\be
\label{geneq}
\dot y + \A y =  b
\ee

\begin{corollary}
\label{crprod.lc1}
Let $y$, $w$ be as above. Then 
$\zeta:= wy$ is a mild solution of
\be
\label{crprod1}
\dot \zeta + \A \zeta = v y + w b.
\ee 
\end{corollary} 

\begin{proof}
Observe that a product of absolutely continuous
functions is absolutely continuous with the usual
formula for the derivative of the product. So, 
given $\varphi \in \dom(\A^*)$, the function $t \mapsto
\la\varphi,\zeta(t) \ra=w(t) \la\varphi,y(t)\ra$
is absolutely continuous and satisfies 
\be
\ddt \la\varphi,\zeta\ra + \la \A^* \varphi,\zeta\ra
= v \la\varphi,y\ra + w \left(\ddt \la\varphi,y\ra
+  \la \A^* \varphi,y\ra \right)
= v \la\varphi,y\ra + w \la\varphi,b\ra 
\ee
meaning that $\zeta$ is solution of \eqref{crprod1}
in a weak sense. The conclusion follows
with Theorem \ref{weaksense-thm}.
\end{proof}

 \begin{theorem}[Basic estimate]
 \label{Gronwall}
There exists $\gamma>0$ not depending on $(f,u)$
such that 
the solution $\Psi$ of \eqref{semg1} satisfies 
 \be
 \label{semest}
\|\Psi\|_{C([0,T];\H)} \leq 
\gamma \left(  \| \Psi_0\|_\H + \|f\|_{L^1(0,T;\H)} + \|\B_1\|_\H \|u\|_1
\right) e^{\gamma \|u\|_1}.
 \ee
 \end{theorem}

 \begin{proof}
 From equation \eqref{semPsi} we get
\be
\begin{aligned}
  \|\Psi(t) \|_\H  \leq & c_\A e^{\lambda_\A t}\|\Psi_0 \|_\H
 + c_\A \int_0^t e^{\lambda_\A(t-s)} \Big( \|f(s)\|_\H + \|\B_1\|_\H |u(s)| \Big) \dd s \\
 & + c_\A e^{\lambda_\A T} \|\B_2\|_{\L(\H)} \int_0^t e^{-\lambda_\A s}  |u(s)| \| \Psi(s)\|_\H \dd s,
 \end{aligned}
 \ee
We conclude with the following Gronwall's inequality:
if $\theta\in L^1(0,T)$ and $a\in L^\infty(0,T)$, then
\be
a(t) \leq \delta + \int_0^t \theta(s) a(s) \dd s
\quad \text{implies}\quad 
a(t) \leq \delta e^{\int_0^t \theta(s) \dd s}.
\ee
 \end{proof}

 The control and state spaces are, respectively,
\be
\calu:=L^1(0,T); \quad \caly:=C(0,T;\calh).
\ee
For $s\in [1,\infty]$ we set 
$\calu_s := L^s(0,T)$.
Let $\uh \in \calu$ be given and $\Psih$ solution of \eqref{semg1}.
The linearized equation at $(\Psih,\uh)$,  to be understood in the 
mild or weak sense, is
\be
\label{semg5lin}
\dot z(t) + \A z(t) = \uh(t) \B_2 z(t) + v(t) (\B_1+\B_2\Psih(t));
\quad z(0)=0,
\ee
where $v\in \U.$ In view of the previous analysis,
for given $v\in \U,$ the equation \eqref{semg5lin} 
has a unique solution that we refer as $z[v].$

\begin{theorem}
\label{diffstateqeu}
The mapping 
$u\mapsto \Psi[u]$ 
(mild solution of \eqref{semPsi})
from $\U$ to $\Y$ is of class $C^\infty$ and we have that 
\be \label{equ:DPsi}
D\Psi[u]v= z[v],\quad \forall v\in \U.
\ee
\end{theorem}

\begin{proof}
In order to prove differentiability of the mapping $u\mapsto \Psi[u],$ we apply the Implicit Function Theorem
to the mapping
 $\cF\colon \U \times \caly
\rightarrow \caly \times \calh$
defined by
\be
\cF(u,\Psi):=\left( 
\Psi - e^{-t\A}\Psi_0 - 
\int_0^t e^{-(t-s)\A} \big( f(s) + u(s) (\B_1+ \B_2 \Psi(s)) \big) \dd s,
\Psi(0) \right).
\ee
This bilinear and continuous mapping is of class $C^\infty$
and it is easily checked that 
$\cF_{\Psi}(u,\Psi)$ 
is an isomorphism, that is, the linear equation 
\be
\label{semg5linz}
z - e^{-t\A}z_0 - 
\int_0^t e^{-(t-s)\A} u(s) \B_2 z(s) \dd s =g,
\quad z(0)=z_0
\ee
has, for any $(g,z_0) \in C(0,T; \calh)\times \calh$,
a unique solution $z$ in $C(0,T; \calh)$,
as can be deduced from the fixed-point argument
in the beginning of the section. 
The conclusion follows.
\end{proof}

\subsection{Regularity of the solution}\label{sec:regsol}

The above result may allow to prove
higher regularity results. 
\begin{definition}[Restriction property]\label{def:restr_prop}
Let $E$ be a Banach space, with norm denoted by
$\|\cdot\|_{E}$ with continuous inclusion in $\H$. 
Assume that the restriction of $e^{-t\A}$ to $E$ 
has image in $E$, and that it is 
a continuous semigroup over this space.
We let $\A'$ denote its associated generator,
and $e^{-t\A'}$ the associated semigroup.
By \eqref{def-domA}-\eqref{def-domAa},
we have that 
\be
\label{def-domAw}
\dom(\A') := \left\{ y\in E; \;\; \lim_{t\downarrow 0} 
\frac{e^{-t\A}y-y}{t} \;\text{belongs to $E$} \right\}
\ee
so that $\dom (\A') \subset \dom (\A)$,
and $\A'$ is the restriction of $\A$ to 
$\dom(\A')$. We have that 
\be
\label{cond:bounded2}
\| e^{-t \A'}\|_{\L(E)} \leq c_{\A'} e^{\lambda_{\A'} t}.
\ee
for some constants 
$c_{\A'}$ and $\lambda_{\A'}$.
Assume that $\B_1 \in E$, 
and denote by $\B'_2$ the restriction of $\B_2$ 
to  $E$, which is supposed to have image in 
$E$ and to be continuous in the topology of $E$,
that is,
\be
\label{lem-reg.l1-1}
\B_1 \in E; \quad \B'_2 \in \L(E). 
\ee
In this case we say that $E$ has the 
{\em restriction property}. 
\end{definition}

\begin{lemma}
\label{lem-reg.l1}
Let $E$ have the restriction property, 
$\Psi_0 \in E$, and $f\in L^1(0,T;E)$ hold.
Then 
$\Psi\in C(0,T;E)$ and the mapping 
$u\mapsto\Psi[u]$ is of class $C^\infty$ from
$L^1(0,T)$ to  $C(0,T;E)$.
\end{lemma}

\begin{proof}
This follows from the semigroup theory 
applied to the generator $\A'$. 
\end{proof} 

\begin{remark}
\label{lem-reg.l1-r}
In view of \cite[Thm. 2.4]{MR710486}
the above Lemma applies 
with $E = \dom(\A)$.
\end{remark}

\subsection{Dual semigroup}\label{sec:absetdual}
Since $\H$ is a reflexive Banach space it is known, e.g. 
\cite[Ch. 1, Cor. 10.6]{MR710486}  
that
$\A^*$ generates another strongly continuous
semigroup called the {\em dual (backward) semigroup} on $\H^*$,
denoted by $e^{-t \A^*}$, which satisfies 
\be
\label{etatransp}
(e^{-t \A})^* = e^{-t \A^*}. 
\ee
Let $(y,p)$ be solution of the
forward-backward system 
\be
\label{yp}
\left\{\ba{rcl}
{\rm (i)} & 
\dot y + \A y &= ay + b,
\\ {\rm (ii)} & 
-\dot p + \A^*p &=a^* p + g,
\ea\right.\ee
where 
\be
\label{yp-bis}
\left\{\ba{l}
b\in L^1(0,T;\H), \\
g\in L^1(0,T;\H^*),\\
a \in L^\infty(0,T;\L(\H)),
\ea\right.\ee
and for a.a. $t\in (0,T)$, $a^*(t) \in \L(\H^*)$ is the adjoint operator of
$a(t)\in \L(\H)$,
so that $a^*\in L^\infty(0,T;\L(\H^*))$.

The mild solutions $y\in C(0,T;\H)$, $p\in C(0,T;\H^*)$
of \eqref{yp}, satisfy
for a.a. $t\in (0,T)$: 
\be
\label{pqform6}
\left\{ \begin{aligned}
{\rm (i)}\,\, & y(t) = e^{-t \A} y(0) +  \int_0^t e^{-(t-s) \A} (a(s)y(s)+b(s)) \dd s,
\\
{\rm (ii)} \,\,&  p(t) = e^{-(T-t) \A^*} p(T) +  \int_t^T e^{-(s-t) \A^*} (a^*(s)p(s)+g(s)) \dd s.
\end{aligned}
\right.
\ee
We have the integration by parts (IBP) Lemma:

\begin{lemma}\label{lem:IBP1}
 \label{pqform}
Let $(y,p) \in C(0,T;\H) \times C(0,T;\H^*) $ satisfy
\eqref{yp}-\eqref{yp-bis}. Then,
\be
\label{pqform0}
\la p(T), y(T)\ra+ \int_0^T \la g(t), y(t)\ra \dd t 
=
\la p(0), y(0)\ra + 
\int_0^T \la p(t), b(t)\ra \dd t.
\ee
\end{lemma}

\begin{proof}
Adding $\int_0^T \la a^*(t)p(t),y(t)\ra \dd t = \int_0^T \la p(t),a(t)y(t)\ra \dd t  $ to both sides of \eqref{pqform0}, we get the equivalent equation
\be
\label{equiv}
\la p(T), y(T)\ra+ \int_0^T \la a^*(t)p(t)+g(t), y(t)\ra \dd t 
=
\la p(0), y(0)\ra + 
\int_0^T \la p(t), a(t)y(t)+b(t)\ra \dd t.
\ee

By \eqref{pqform6}(i), we have the following expression for the first term in the l.h.s. of \eqref{equiv}:
\be
\label{pqform2}
\la p(T), y(T)\ra = 
\la e^{-T \A^*} p(T),y(0)\ra 
+ \int_0^T \la e^{-(T-s) \A^*} p(T), a(s)y(s) + b(s)\ra \dd s.
\ee
Similarly, for the integrand in the second term in the l.h.s. of \eqref{equiv} we get, in view of \eqref{pqform6}(i),
\be
\label{pqform3}
\begin{aligned}
&\la a^*(t)p(t)+ g(t), y(t)\ra \\
&= 
\la e^{-t \A^*} (a^*(t)p(t)+g(t)),y(0) \ra
\\ & \hspace{3mm}
+ \int_0^t \la e^{-(t-s) \A^*} (a^*(t)p(t)+g(t)),a(s)y(s)+ b(s) \ra \dd s.
\end{aligned}
\ee
Adding \eqref{pqform2}  and \eqref{pqform3}, and regrouping the terms we get
\be
\label{pqform4}
\la p(T), y(T)\ra + \int_0^T \la a^*(t)p(t)+ g(t), y(t)\ra \dd t = R_1+R_2,
\ee
where
\be
R_1 :=\la e^{-T\A^*} p(T),y(0) \ra+
\la  \int_0^T  e^{-t \A^*} (a^*(t)p(t)+g(t)) \dd t,y(0) \ra=\la p(0),y(0) \ra,
\ee
and $R_2$ is the remainder. Thanks to Fubini's Theorem
\be
\label{R2}
\begin{aligned}
 R_2
&=
 \int_0^T \la e^{-(T-s) \A^*}p(T) +\int_s^T e^{-(t-s)\A^*}(a^*(t)p(t)+g(t)) \dd t, a(s)y(s)+ b(s) \ra \dd s\\
 &= \int_0^T \la p(s), a(s)y(s)+ b(s) \ra \dd s
 \end{aligned}
 \ee
 From \eqref{pqform4}-\eqref{R2} we get \eqref{equiv}.
The result follows. 
\end{proof}

\begin{corollary}\label{lem:int_by_parts}
Let $(y,p)$ be as in Lemma \ref{pqform}, and $\varphi$ an absolutely continuous function over $(0,T)$. 
Then
\be
\label{lem:int_by_parts-1}
\int_0^T \dot \varphi (t) \la p(t),y(t) \big\ra \dd t
= \big[ \varphi (t) \la p(t),y(t) \ra \big]_0^T 
- \int_0^T \varphi (t) \Big( \la p(t),b(t)\ra - \la g(t),y(t)\ra  \Big) \dd t.
\ee
\end{corollary}

\begin{proof}
By the IBP Lemma \ref{pqform}, replacing $T$ by an arbitrary 
time in $(0,T)$, we see that $h(t) := \la p(t), y(t)\ra$
is a primitive of the integrable function 
$\la p(t), b(t)\ra - \la g(t), y(t)\ra$. 
The Corollary follows then from the integration by parts formula in 
the space of absolutely continuous functions. 
\end{proof}

Given $(y,p)$ solution of \eqref{yp} and 
$B\in \L(\calh)$, set 
$\Phi(t) :=B y(t)$. 
Then $\Phi\in L^\infty(0,T;\H)$, is solution of an equation involving the {\em  operator} $ \A B - B\A.$ 
In order to defined properly the latter,
consider the following hypotheses:
\be
\label{commutator_identity-hyp}
\left\{ \ba{lll}
{\rm (i)}  \;\; 
B \dom(\A) \subset \dom(\A); \\
{\rm (ii)}  \;\; 
B^* \dom(\A^*) \subset \dom(\A^*).
\ea\right.\ee
Whenever these hypotheses hold,  we may define the operators
below, with domains $\dom(\A)$ and $\dom(\A^*),$ respectively:
\be
\left\{
\begin{split}
[\A,B] &:=\A B-B\A ,\\ 
[B^*,\A^*]&:=B^*\A^* - \A^*B^*.
\end{split}
\right.\ee
Let $E$ be a subspace  of $\H$ with norm denoted by
$\|\cdot\|_E$, and continuous
inclusion. Consider the following 
{\em bracket extension property}
\be
\label{hypEspace-EB}
\left\{ \ba{lll}
\dom(\A)  \subset E \subset \calh.
\\
\text{$[\A,B]$ has an extension by continuity over $E$,
say $\overline{[\A,B]}$.}
\ea\right.
\ee

\begin{proposition}
\label{commutator_identity.p}
{\rm (i)}
Let \eqref{commutator_identity-hyp}
hold, and consider $(y,\phi)\in \dom(\A)\times \dom(\A^*)$.
Then
$y\in \dom([B^*,\A^*]^*)$, 
$\phi\in \dom([\A,B]^*)$,  
and we have that 
\be
\label{commutator_identity}
\la \phi, [\A,B]y \ra 
=\la [B^*,\A^*]\phi,y \ra 
=\la [\A,B]^*\phi, y \ra
=\la \phi,[B^*,\A^*]^*y \ra.
\ee
{\rm (ii)}
Let in addition \eqref{hypEspace-EB} hold. Then 
\be
\label{commutator_identity2}
\la [B^*,\A^*]\phi,y \ra=\la \phi , \overline{[\A,B]}y \ra,
\quad \text{for all $y\in E$ and $\phi\in \dom(\A^*)$.}
\ee

\end{proposition}
\begin{proof}
 (i)
We have that 
 \be
\ba{lll}
\la \phi , [\A,B]y \ra
&=
\la \phi ,\A By \ra - \la \phi , B\A y \ra 
 = 
\la \A^*\phi, By \ra - \la B^*\phi , \A y \ra 
\\ & = 
\la B^*\A^*\phi, y \ra - \la \A^*B^*\phi , y \ra 
 = 
\la [B^*,\A^*]\phi, y \ra
\ea\ee
proving the first equality in
\eqref{commutator_identity}. 
This equality implies that 
$\phi\in \dom([\A,B]^*)$ as well as 
the second equality (by the definition of the adjoint). 
We obtain the last equality by similar arguments. 
\\ (ii)
Let $(y_k) \subset \dom(\A)$, $y_k\rar y$ in $E$. 
Then \eqref{commutator_identity}
holds for $y_k$, and passing to the limit in the first equality
we get \eqref{commutator_identity2}. 
\end{proof}

\begin{remark}
We do not have in general 
$[\A,B]^* = [B^*,\A^*]$ since the l.h.s has a domain which may be
larger than the one of $\A^*$. 
\end{remark}

Let us define $M\in \L(E,\H)$ by 
\be
\label{def-m-oper}
M y := \overline{[\A,B]} y,
\ee
so that 
$M^* \in \L(\H^*,E^*)$. 

\begin{corollary}
\label{lem:int_by_parts-1.c}
Let  \eqref{commutator_identity-hyp}
and \eqref{hypEspace-EB} hold, $(y,p)$
be solution of \eqref{yp}-\eqref{yp-bis}, and
$\varphi$ be an absolutely continuous function over $(0,T)$. 
{\rm (i)}
 Let $y\in L^1(0,T;E)$. Then
$\Phi(t)=B y(t)$ is a mild solution of
\be
\label{ypphi}
\dot \Phi + \A \Phi 
=  
B(ay+b) + My
=  
a \Phi + Bb + [B,a] \Phi + My,
\ee
and we have that 
\be
\label{lem:int_by_parts-2}
\begin{aligned}
\int_0^T &\dot \varphi (t) \la p(t),\Phi (t) \big\ra \dd t
 =   \disp
\big[ \varphi (t) \la p(t),\Phi(t) \ra \big]_0^T 
\\  &
- \int_0^T \varphi (t) \Big( \la p(t),B b + [B,a] y + 
M y (t)\ra - \la g(t),\Phi(t)\ra  \Big) \dd t.
\end{aligned}
\ee
{\rm (ii)} 
Assume that 
$E$ has the restriction property, and that
$M^* p\in L^1(0,T;\H^*)$. 
Then the following
IBP formula holds:
\be
\label{lem:int_by_parts-2tr}
\begin{aligned}
\int_0^T &\dot \varphi (t) \la p(t),\Phi (t) \big\ra \dd t
 =   \disp
\big[ \varphi (t) \la p(t),\Phi(t) \ra \big]_0^T 
\\  &
- \int_0^T \varphi (t) \Big( \la p(t),B b + [B,a] y \ra
+ 
\la M^* p(t), y (t) \ra - \la g(t),\Phi(t)\ra  \Big) \dd t.
\end{aligned}
\ee
\end{corollary}

\begin{proof}
(i)
By Theorem \ref{weaksense-thm},
 it suffices to prove that $\Phi$
is a weak solution of \eqref{ypphi}.
Let $\phi\in \dom(\A^*)$ and set $f:=ay+b$. Then 
$\la \phi ,  \Phi (t) \ra =\la B^*\phi , y(t) \ra$ is absolutely continuous, 
and so, by \eqref{weak_ball} and the previous Proposition:
\be
\label{transbaphi}
\ba{lll }\disp
\ddt \la \phi , \Phi(t) \ra &=& \disp
\ddt \la  B^*\phi, y(t) \ra =  - \la \A^*B^*\phi , y(t)\ra + \la B^*\phi , f \ra
\\ & = &\disp
-\la B^*\A^* \phi, y(t) \ra  + 
\la [B^*,\A^*] \phi ,y(t) \ra + \la \phi , B f \ra
\\ & =& \disp
-\la  \A^* \phi , \Phi(t) \ra  + 
\la [B^*,\A^*] \phi , y(t) \ra + \la \phi , B f \ra
\\ & =& \disp
-\la \A^* \phi , \Phi(t) \ra  +M y(t)  + \la \phi , B f \ra,
\ea\ee
where we use 
Proposition \ref{commutator_identity.p}(ii)  in the last equality.
Point (i) follows.
\\ (ii)
Let $y_{0k}$ in $E$ converge to $y_0$ in $\H$, 
and $b_k \in L^1(0,T;E)$, $b_k\rar b$ in $L^1(0,T;\H)$.
Since $E$ has the restriction property, 
the associated $y_k$ belong to $C(0,T;E)$ and therefore
 \eqref{lem:int_by_parts-2}
holds for $(b_k,y_k)$. 
Since $M\in \L(E, \H)$ we have that 
\be
\int_0^T \varphi (t) \la p(t), M y(t)\ra \dd t  
=
\int_0^T \varphi (t) \la M^* p(t), y(t)\ra_E \dd t  
=
\int_0^T \varphi (t) \la M^* p(t), y(t)\ra_H \dd t  
\ee
where in the last equality we use the fact 
that $M^*p \in L^2(0,T;\H)$, and that 
since $E$ is a subspace of $\H$ with dense inclusion,
the action of $\H^*$ over $E$ can be identified to the
duality pairing in $\H$.
So, \eqref{lem:int_by_parts-2tr}  holds with $(b_k,y_k)$. 
Passing to the limit in the latter we obtain the conclusion. 
\end{proof}

\subsection{The optimal control problem}\label{sec:notation}

Let $q$ and $q_T$ be continuous quadratic forms over $\H$,
with associated symmetric and continuous 
operators 
\be
 Q, Q_T \in \L(\H,\H^*); \;\; 
q(y) :=\la Qy,y\ra; \;\; q_T(y) :=\la Q_Ty,y\ra.
\ee
Given
\be
\label{psi_d-psi-dt}
\Psi_d\in L^\infty(0,T;\H); \quad \Psi_{dT}\in\H,
\ee
we introduce the cost function
\be
\label{def-cost-fun}
J(u, \Psi) := 
\alpha \int_0^T  u(t) \dd t 
+
\half\int_0^T q( \Psi(t)-\Psi_d(t) )\dd t 
+
 \half q_T(\Psi(T)-\Psi_{dT})
\ee
with $\alpha \in \cR$.
The reduced cost is
\be
F(u) := J(u,\Psi[u]). 
\ee
The set of feasible controls is 
\be
\label{Uad}
\calu_{ad}:= \{ u\in\calu; \; u_m \leq u(t) \leq u_M \; \text{a.e. on } [0,T] \},
\ee
with $u_m<u_M$ given constants. 
The optimal control problem is 
\be\label{problem:OC}
\tag{P}
\Min_u F(u); \quad u\in \calu_{ad}.
\ee
We say that $\uh \in \calu_{ad}$ is a 
{\em minimum (resp. weak minimum)} of problem \eqref{problem:OC}
if $F(\uh) \leq F(u)$, for any 
$u\in \calu_{ad}$ 
(resp. $u\in \calu_{ad}$,
sufficiently close to $\uh$ in the norm
of $L^\infty(0,T)$). 

Given 
$(f,y_0)\in L^1(0,T;\H)\times \H$, 
denote by $y[y_0,f]$
the mild solution of 
\be
\dot y(t) + \A y(t)= f (t),\quad
t\in (0,T), 
\qquad y(0)=y_0. 
\ee
The {\em compactness hypothesis} is
\be
\label{compachyp}
\left\{ \ba{lll}
\text{For given $y_0\in \calh$,
the mapping $f\mapsto \B_2 y [y_0,f]$}
\\
\text{is compact from
$L^2(0,T;\calh)$ to $L^2(0,T;\calh)$.}
\ea\right.\ee

\if { 
Often we will need to assume that
\be
\label{compachypw}
\text{Either $\calb_2=0$ or the 
compactness hypothesis \eqref{compachyp}
holds.}
\ee
} \fi 

\begin{lemma}
\label{compactlemu}
Let \eqref{compachyp} hold. 
Then the mapping 
$u\mapsto \Psi[u]$
is sequentially
continuous  from $\calu_\infty$ endowed with the 
weak$*$ topology, to $C(0,T;\H)$  endowed with the 
weak topology.
\end{lemma}

\begin{proof}
If $\calb_2=0$, the mapping $u\mapsto \Psi[u]$
is linear continuous,
and therefore weakly continuous from
 $\calu_\infty$ to  $C(0,T;\H).$ 
 
Otherwise,  for a bounded sequence $(u_k)$ in 
$\calu_\infty$ and associated sequence of states $(\Psi_k)$,  
extracting if necessary a subsequence, we have that 
$(u_k)$ weakly$*$ converges to some $\tilde u$
in $\calu_\infty,$ and 
$\Psi_k$ strongly converges in $L^2(0,T;\H)$ to some
$\tilde \Psi$, so that 
$u_k \calb_2 \Psi_k$ weakly converges in 
$L^2(0,T;\H)$ to $\tilde u\calb_2 \tilde\Psi.$ Hence, by the 
expression of mild solutions, 
$\Psi_k$ weakly converges in $C(0,T;\H)$  to $\tilde \Psi$
and $\tilde\Psi$ is the state associated with $\tilde u$. 
\end{proof}

\begin{theorem}
\label{existsol}
Let \eqref{compachyp} hold.  
Then problem $(P)$ has a nonempty set of minima. 
\end{theorem}

\begin{proof}
Let us first notice that the problem is feasible. 
Since $\calu_{ad}$
is a bounded subset of $\calu$,   any 
minimizing sequence $(u_k)$ has a weakly$*$
 converging  subsequence to some $\tilde u\in \calu$. 
Reindexing, we may assume that $(u_k)$ weakly$*$
converges to $\tilde u$. So $(u_k)$ also weakly converges
to $\tilde u$ in $L^2(0,T)$. Since $\calu_{ad}$
is a closed subset of $L^2(0,T)$, necessarily  $\tilde u\in \calu_{ad}$.
By Lemma \ref{compactlemu}, 
$\Psi[u_k]\rar \Psi[\tilde u]$ weakly in $L^2(0,T;H)$.
Since $J$ is convex and continuous in
$L^2(0,T)\times L^2(0,T;\calh)$, it is weakly l.s.c.
so that $J(\tilde u,\Psi[\tilde u]) \leq \lim_{k\rightarrow \infty} J(u_k,\Psi[u_k]).$ Since the limit in the right hand-side of latter inequality is the optimal value, necessarily $(\tilde u,\Psi[\tilde u])$ is optimal.
The result follows.
\end{proof}

The costate equation is
\be
\label{semg8}
- \dot p + \A^* p = Q(\Psi-\Psi_d) + u \B_2^* p; 
\quad
p(T) = Q_T(\Psi(T)-\Psi_{dT} ).
\ee
We denote by $p[u]$ its mild (backward) solution:
\be
\label{semadj}
p(t) = e^{ (t-T) \A^*} Q_T(\Psi(T)-\Psi_d(T) )
+ \int_t^T e^{ (t-s) \A^*} 
\big( Q(\Psi(s)-\Psi_d(s) ) + u(s) \B_2^* p(s) \big) \dd s.
\ee
We set
\be
\label{expr-lambda}
\Lambda(t):=\alpha + \la p(t),\B_1+\B_2\Psih(t) \ra.
\ee

\begin{theorem}
\label{diffcostfct}
The mapping $u\mapsto F(u)$ is of class 
$C^\infty$ from $\U$ to $\cR$ and we have that 
\be
\label{diffcostfct1}
DF(u)v=\int_0^T \Lambda(t)  v(t) \dd t,\qquad \text{for all } v\in \U.
\ee
\end{theorem}

\begin{proof}
That $F(u)$ is of class $C^\infty$
follows from Theorem \ref{diffstateqeu}
and the fact that $J$ is of class $C^\infty$. This also implies
that, setting $\Psi:=\Psi[u]$ and $z:=z[u]$: 
$$
DF(u)v = 
\alpha \int_0^T v(t) \dd t
+
\int_0^T Q( \Psi(t)-\Psi_d(t), z(t) ) \dd t 
+
Q_T( \Psi(T)-\Psi_{dT}, z(T) ).
$$
We deduce then \eqref{diffcostfct1} from Lemma \ref{pqform}.
\end{proof}

Let for $u \in \Uad$ and $I_{m}(u)$ and $I_{M}(u)$ be the associated contact sets defined, up to a zero-measure set, as
\be
\left\{ 
\begin{aligned}
  I_{m}(u) &:= \{t\in (0,T) : u(t) = u_m\},
\\
   I_{M}(u) &:= \{t\in (0,T) : u(t) = u_M \}.
\end{aligned}
\right.
\ee
The first order optimality necessary condition is given as follows.
\begin{proposition}
\label{Prop1order}
Let $\uh$ be a weak minimum of \eqref{problem:OC}.
Then, up to a set of measure zero, there holds
 \be
 \{t;\; \Lambda(t)>0\}\subset I_m(\uh),\quad \{t;\; \Lambda(t)<0\}\subset I_M(\uh).
 \ee
\end{proposition}
\begin{proof}
 $F$ is differentiable and attains its minimum over the convex set
 $\Uad$ at $\uh$ and thus,
if $\uh+v\in \Uad$, then 
 \be
 0 \le \lim_{\sigma \downarrow 0} \frac{F(\uh+\sigma v)-F(\uh)}{\sigma}=DF(\uh)v.
 \ee
 Since $DF(\uh)v=\int_0^T \Lambda (t) v(t)\dd t $, 
this means that
\be
\int_0^T \Lambda (t) (u(t)-\uh(t) ) \dd t \geq0, 
\quad \text{for all $u\in \Uad$,}
\ee
from which the conclusion easily follows.
\end{proof}

Set $\delta \Psi:= \Psi-\hat\Psi.$ 
We note for future reference that,
since $u \Psi-\uh\Psih= u \delta\Psi + v \Psih$,
we have that $\delta \Psi$
is the mild solution of:
\be
  \label{semdeltaPsi}
  \ddt \delta \Psi(t)+ \A \delta \Psi(t) = 
u(s)\B_2 \delta\Psi(s) + v(t) (\B_1 + \B_2 \Psih(t) ).
\ee
Thus, $\eta := \delta\Psi-z$ is solution of 
\be
  \label{sem-eta}
  \dot \eta(t) + \A  \eta(t) = \uh \calb_2 \eta(t) + v(s)\B_2 \delta\Psi(s).
\ee

We get the following estimates.
\begin{lemma}
\label{Lemmaestz}
The linearized state $z$ solution of \eqref{semg5lin},
the solution $\delta \Psi$ of \eqref{semdeltaPsi}, and 
$\eta = \delta\Psi-z$ solution of \eqref{sem-eta} satisfy,
whenever $v$ remains in a bounded set of $L^1(0,T)$:
\begin{eqnarray}
\label{estz}
\|z\|_{L^\infty(0,T;\H)} &=& O( \|v\|_1),
\\
\label{estdeltaPsi}
\|\delta\Psi\|_{L^\infty(0,T;\H)} &=& O( \|v\|_1),
\\
\label{esteta}
\|\eta\|_{L^\infty(0,T;\H)} &=& O(\| \delta\Psi\,v\|_{L^1(0,T;\H)}) = O(\|v\|_1^2).
\end{eqnarray}
\end{lemma}

\begin{proof}
By arguments close to those in the proof of 
Theorem \ref{Gronwall}, we get 
\be
\|z\|_{L^\infty(0,T;\H)} \leq 
\gamma'   \|v\|_1 e^{\gamma' \|v\|_1} 
\ee
for some $\gamma'$ not depending on $v$,
which, since $\|v\|_1$ is bounded, 
proves \eqref{estz}. Then, we also have by
\eqref{semdeltaPsi}
\be
\|\delta\Psi\|_{L^\infty(0,T;\H)} \leq K(\|u\|_1,\|\B_2\|_{\L(\H)})\, (\|\B_1\|_\H + \|\B_2\|_{\L(\H)} \|\hat\Psi\|_{L^\infty(0,T;\H)}) \|v\|_1,
\ee
which implies \eqref{estdeltaPsi}. Finally, it holds with 
\eqref{sem-eta}
\be
\begin{aligned}
\|\eta\|_{L^\infty(0,T;\H)} &\leq K\big(\|\uh\|_1,\|\B_2\|_{\L(\H)}\big) \|v \B_2 \delta\Psi\|_{L^1(0,T;\H)} \\
& \leq  K\big(\|\uh\|_1,\|\B_2\|_{\L(\H)}\big) \|\B_2\|_{\L(\H)}\| \delta\Psi\,v\|_{L^1(0,T;\H)},
\end{aligned} 
\ee
that yields the first equality in \eqref{esteta}. The second one follows in view of \eqref{estdeltaPsi}. 
\end{proof}

\section{Second order optimality conditions}\label{sec:soocg}
\subsection{A technical result}
Let $\uh \in \calu$ , 
with associated state $\Psih=\Psi[\uh]$ and costate
$\ph$ solution of \eqref{semadj}, 
$v\in L^1(0,T)$, and $z\in C(0,T;\H)$.
Let us set
\be
  \label{tildeQ}
  \Q(z,v) := \int_0^T \Big( q (z(t) )
 + 2 v(t) \la \ph(t),\B_2  z(t) \ra  \Big)\dd t + q_T (z(T) ).
  \ee

\begin{proposition}\label{ExpEasy}
  Let $u$ belong to $\calu$. 
Set $v:=u-\uh$, 
  $\Psih := \Psi [\uh]$, $\Psi := \Psi [u]$.
Then
  \be
  F(u) = F(\uh) + D F(\uh)v + \half  \Q(\delta\Psi,v).
  \ee
  
\end{proposition}

\begin{proof}
 We can expand the cost function as follows:
\be
\label{semcost}
  \begin{aligned}    
& F(u)= 
F(\uh) + \half ( q(\delta\Psi) + q_T(\delta\Psi(T) ) )
\\ & 
+ \alpha \int_0^T   v(t) \dd t
+ \int_0^T 
Q( \hat \Psi(t) -\Psi_d(t), \delta \Psi) ) \dd t  
+ Q_T( \hat\Psi(T) -\Psi_d(T),\delta \Psi(T)).
  \end{aligned}
\ee
Applying Lemma \ref{lem:IBP1}
to the pair $(z,\ph)$, where $z$ is solution of the 
linearized equation \eqref{semg5lin},
and using the expression of $\Lambda$ in 
\eqref{expr-lambda}, we obtain the result.
\end{proof}

\begin{corollary}\label{cor:expansion2}
Let $u$ and $\uh$ be as before, and set
$z := z[v]$. Then 
\be
F(u) = F(\uh) + D F(\uh)v + \half  \Q(z,v) + O(\|v\|^3_1).
\ee 
\end{corollary}

\begin{proof}
We have that
\begin{multline}
\Q(\delta\Psi,v) -\Q(z,v) = \int_0^T
Q(\delta \Psi(t) +z(t),\eta(t))
+ 2 v(t) \la p(t), B_2 \eta(t) \ra \dd t \\+ 
Q_T (\delta \Psi(T)+z(T),\eta(T)).
\end{multline}
By \eqref{estz}-\eqref{esteta}
we have that
\be
 \| \delta \Psi \|_{L^\infty(0,T;\calh)} +
 \| z \|_{L^\infty(0,T;\calh)} = O( \|v\|_1 ),
\ee
\be
\| \eta\|_{L^\infty(0,T;\calh)} = O( \|v\|_1 \|\delta\Psi\|_{L^\infty(0,T;\calh)} )
= O( \|v\|^2_1 ).
\ee
The result follows.
\end{proof}

Note that we will derive a refined Taylor expansion in 
Proposition \ref{prop:expansion_w}.

\subsection{Second order necessary optimality conditions}\label{sec:soonoc-n}
Given a feasible control $u$, the critical cone is defined as
\be
C(u) :=
\left\{
\begin{aligned}
  & v\in L^1(0,T) \,|\ \Lambda(t) v(t)=0\text{ a.e. on } [0,T], \\
  & v(t)\ge 0\,\, \text{a.e. on } I_m( u),\; v(t)\le 0\text{ a.e. on }I_M( u)
\end{aligned}\right\}.
\ee

\begin{theorem} \label{thm:nec_sec_ord_cond}
Let $\uh$ be a weak minimum 
of \eqref{problem:OC}. Then there holds,
  \be
  \Q(z[v],v) \ge 0\quad \text{for all } v \in C(\uh).
  \ee
\end{theorem}

\begin{proof}
  Let $v \in C(\uh)$ with $v \neq 0$. For 
$0<\eps<u_M-u_m$, we set
  \be
    v_{\varepsilon}(t):=
    \left\{
    \begin{array}{rl}
      0, &\text{if } \uh(t) \in (u_m,u_m+\varepsilon)\cup (u_M -
      \varepsilon, u_M), 
\text{ or } |v(t)| > 1/\eps,\\
      v(t), &\text{otherwise}.
    \end{array}
    \right.
  \ee
Then $DF(\uh)v_{\eps} =0$, and for 
  $\sigma \in (0,  \varepsilon^2)$,
we have that 
 $\uh + \sigma v_{\varepsilon}\in \cU_{\text{ad}}$.
  Hence, from Corollary~\ref{cor:expansion2}, we get for $z_{\varepsilon} := z[v_{\varepsilon}]$
  that
  \begin{align}
    \label{ineq:Q}
    0 \le 2 \lim_{\sigma \to 0} \frac{F(\uh + \sigma v_{\varepsilon}) - F(\uh)}{\sigma^2}=
    \Q(z_{\varepsilon},v_{\varepsilon}).
  \end{align}
  Since $v_\eps \to v$ in $L^1(0,T)$ when $\eps \rightarrow 0,$ then we obtain from Lemma
  \ref{compactlemu} that $ z_{\varepsilon} \rightarrow  z[v]$ in $C(0,T;\calh)$ and the assertion follows from \eqref{ineq:Q} and the continuity of~$\Q$.
\end{proof}

\subsection{Principle of Goh transform}\label{sec:csggt}
\subsubsection{Goh transform }
We now introduce the Goh transform on differential equations 
and on quadratic forms. We 
need to perform variants of it for equations 
\eqref{semdeltaPsi}-\eqref{sem-eta}
satisfied by 
$\delta \Psi$ and $\eta$. 
So, we consider a general setting.
Next let $y$ be the mild solution of
\be
\label{geneqs} 
\dot y + \A y = a y + b^0v,\qquad y(0)=0,
\ee
with
\be \label{geneqs_data1}
a \in L^\infty (0,T; \L(\calh)); \quad b^0 \in C(0,T;\calh),
\ee
and $b^0$ is a mild solution of 
\be  \label{geneqs_data2}
\dot b^0 + \A b^0 = g^0 \in L^2(0,T;\calh).
\ee
Given $v\in L^1(0,T)$ and $y$ the corresponding solution of
\eqref{geneqs}, let us consider the {\em Goh transform} associated
with \eqref{geneqs} as the mapping that, 
given $(a,b^0,g^0)$,  associates to the pair $(v,y)$
the pair 
$(w,\xi_y) \in AC(0,T) \times C(0,T;H)$ 
defined by
\be\label{GohT}
w(t) := \int_0^t v(s) \dd s,\quad \xi_y := y-wb^0.
\ee
We set 
$b^1 := ab^0 - g^0$
and note that the norms below are well-defined:
\be
\|a\|_\infty:= \|a\|_{L^\infty(0,T;\L(\H))}; \quad
\|b^i\|_s := \|b^i\|_{L^s(0,T;\H)}; \quad i=0,1; \; s\in [1,\infty],
\ee
although using the same notation for different norms, there is no danger of confusion.
In view of Corollary \ref{crprod.lc1} 
and Theorem \ref{Gronwall} we get:

\begin{lemma}
\label{eq-smgs}
Let \eqref{geneqs}-\eqref{geneqs_data2} hold.  
Then $\xi_y$ is the mild solution of 
\be
\label{eqxidefg}
\dot\xi_y + \A \xi_y  =  a\xi_y + w b^1; \quad \xi(0)=0. 
\ee
In addition there exists $c:\RR_+\rar\RR_+$
nondecreasing such that the constant 
$c_a := c(\|a\|_\infty )$ satisfies 
\begin{gather}
\label{estytx}
\| \xi_y \|_{C([0,T];\calh)} \leq c_a   \|b^1\|_2 \|w\|_2,\\
\label{esty}
\|y\|_2 \leq \left( 
T^{1/2}  c_a   \|b^1\|_2 + \|b^0\|_\infty \right) \|w\|_2.
\end{gather}
\end{lemma}

\begin{proof}
By the semigroup theory there exists $c:\RR_+\rar\RR_+$
nondecreasing such that 
\be
\begin{aligned}
\| \xi_y \|_{C([0,T];\calh)} 
\leq c(\|a\|_\infty)   \|b^1\, w\|_{L^1(0,T;\H)}
\leq  c(\|a\|_\infty ) \|b^1\|_2 \|w\|_2,
\end{aligned}
\ee 
so that \eqref{estytx} holds.
Since $y=\xi_y+wb^0$, we get 
\be
\|y(t)\|_\H \leq c(\|a\|_\infty)   \|b^1\|_2 \|w\|_2 + \|b^0\|_\infty
|w(t)|,\qquad \text{for a.a. } t\in (0,T),\\
\ee
implying \eqref{esty}.
\end{proof}

\if{
\begin{remark} 
\label{eq-smgs-r}
Let $E$ be a subspace of $\H$ satisfying the restriction property for
the semigroup generated by $\A$,  and such that 
the restriction $a'(t)$ of $a(t)$ to $E$ belongs to 
$L^\infty(0,T;\L(E))$ and $b^1\in L^2(0,T;E)$.
Then we have similar estimates as above with 
$E$ instead of $\H$. That is, 
there exists $c':\RR_+\rar\RR_+$
nondecreasing such that 
$c'_a := c'(\|a\|_{L^\infty(0,T;\L(\H))})$ satisfies 
\begin{gather}
\label{estytxe}
\| \xi_y \|_{C([0,T];E)} \leq c'_a   \|b^1\|_{L^2(0,T;E)}  \|w\|_2,\\
\label{estye}
\|y\|_{L^2(0,T;E)} \leq \left( 
T^{1/2}  c_a   \|b^1\|_{L^2(0,T;E)}  + \|b^0\|_{L^\infty(0,T;E)}  \right) \|w\|_2.
\end{gather}
\end{remark}
} \fi

\begin{remark}
The Goh transform has the same structure 
as in the ODE case (see e.g. equations (27)-(30) in \cite{ABDL12}).
 In fact, if we write the equation \eqref{geneqs} in the form
$
\dot y = (a-\A)y+b^0 v,$
in view of \eqref{eqxidefg}, $\xi_y$ defined by Goh transform \eqref{GohT} is solution of 
$
\dot\xi_y = (a-\A) \xi_y + w {b}^1,
$
with $b^1 = (a-\A)b^0 - \dot b^0 = ab^0 - g^0$.
\end{remark}

We assume the existence of 
$E_1\subset \H$ with continuous inclusion
having the restriction property, and such that
\be
\dom(\A) \subset E_1.
\ee
We can use $\B_2$ to denote the restriction of $\B_2$ to $E_1,$  with no risk of confusion, and let us write $\B^k_i$ to refer to $(\B_i)^k.$
In the remainder of the paper we make the following hypothesis:
\be
\label{hyp-Goh-tr1}
\left\{ \ba{lll}
{\rm (i)} &
\B_1 \in \dom (\A),\;  
 \\ {\rm (ii)} &
\B_2 \dom(\A) \subset \dom(\A),\quad 
\B_2^* \dom(\A^*) \subset \dom(\A^*),\;\; 
\\
{\rm (iii)} &
\text{for $k=1,2:$ 
$\left[\A, \B^k_2\right]$ has a continuous
extension to $E_1$,}
\\
& \text{denoted by $M_k$, 
}
\\ {\rm (iv)} &
f \in L^\infty(0,T;\H);  \quad 
M^*_k \ph\in L^\infty(0,T;\H^*), \; k=1,2,
\\ {\rm (v)} &
\Psih\in L^2(0,T;E_1); \;\; 
[M_1,\B_2] \Psih \in L^\infty(0,T;\H).

\ea\right.
\ee
We refer to Section \ref{sec:applschr},
where examples of problems, where these hypotheses are easily
checked, are provided.

\begin{remark}
Observe that  \eqref{hyp-Goh-tr1} (ii) implies that
\be
\B^k_2 \dom(\A) \subset \dom(\A),\quad 
(\B^k_2)^* \dom(\A^*) \subset \dom(\A^*),\qquad \text{for } k=1,2.
\ee
So, $[\A, \B_2]$ 
is well-defined as operator
with domain $\dom(\A)$, and point (iii) makes sense.
\end{remark}

\subsubsection{Goh transform for $z$}\label{sec:Goh1}
Let $\uh\in \calu$ have associated state $\Psih$.
Recall that $z$ is solution of the linearized state equation 
\eqref{semg5lin}. Set
\be
\calb(t) := \B_1 + \B_2 \Psih(t).
\ee 
We apply Corollary \ref{lem:int_by_parts-1.c}
with $B:= \B_2$ and $y:= \Psih$, so that
$(a,b) = (\uh\B_2,f+\uh \B_1)$. 
Then $\Phi:= \B_2 \Psih$ satisfies
\be
\label{phi-zz}
\dot \Phi + \A \Phi 
= 
\B_2 ( f + \uh \B) + M_1 \Psih.
\ee
Setting $\Phi':= \B=\B_1 + \Phi,$ 
thanks to \eqref{hyp-Goh-tr1}, we get
\be
\label{dphipaphip}
\dot \Phi' + \A \Phi'
= g_z, \quad
\text{where } g_z :=
 \A \B_1 + \B_2 ( f + \uh \B) + M_1  \Psih.
\ee
We next apply Lemma \ref{eq-smgs} to 
the linearized state equation \eqref{semg5lin},
with here the
pair $(a,b)$ corresponding to 
$(a_z,b_z)= (\uh \B_2,\calb)$.
Clearly $a_z\in L^\infty(0,T;\call(\H))$,
$b_z\in C(0,T;\H)$, and by 
\eqref{dphipaphip}, we have that 
$\dot \Phi' + \A \Phi'$ belongs to 
$L^2(0,T;\H)$ in the sense of mild solutions. 
It follows that the dynamics for 
$\xi:=z-w\B$ with $w:=\int_0^t v(s)\dd s$ reads
\be
\label{diffeqxi}
\dot\xi + \A \xi  = \uh \B_2 \xi + w b^1_z; 
\ee
where
\be
\label{equ-bunz}
b^1_{z} 
= a_z b_z - g_z = 
-\B_2 f -  M_1 \Psih  - \A \B_1.
\ee

\begin{proposition}\label{prop:estimatez}
The solution $z$ of the linearized state equation \eqref{semg5lin} satisfies the following estimate
\be 
\label{est1}
\|\xi\|_{C(0,T;\H)} + \|z\|_{L^2(0,T;\H)}
=   O\big(  \|w\|_2 \big).
\ee

\end{proposition}

\begin{proof}
This follows from the restriction property and since
hypothesis \eqref{hyp-Goh-tr1} guarantees that $b^1_z \in L^\infty(0,T;\H)$. 
\end{proof}

\subsection{Goh transform of the quadratic form} 
Let again $\uh\in\calu$, and set 
$\Psih=\Psi[\uh]$ and $\ph=p[\uh]$. 
We recall the definition of the operator $M$ in
\eqref{def-m-oper}. 
Consider the space
\be
\label{def--ww}
W := \left(L^2(0,T;E_1) \cap C([0,T];\H) \right)
\times L^2(0,T) \times \cR.
\ee
We introduce the continuous quadratic form over $W,$ defined by 
\be
\label{Omega}
\hatQ(\xi,w,h) = \hatQ_T(\xi,h) +\hatQ_a(\xi,w)+\hatQ_b(w),
\ee
where 
$\hatQ_b(w):= \int_0^T w^2(t) R(t)\dd t$ and 
\begin{align}\label{OmegaT}
\hatQ_T(\xi,h)&:= q_T(\xi(T) + h \calb(T) )
 + h^2 \la \ph(T),\B_2\B_{1}+\B_2^2 \Psih(T)\ra +h \la \ph(T),\B_2 \xi(T) \ra,\\
\hatQ_a(\xi,w)&:= \int_0^T \Big( q(\xi) + 2 w  \la Q\xi,\B \ra
 + 2 w\la  Q(\Psih-\Psi_d),\B_2\xiz\ra - 
2w \la M^*_1 \ph, \xiz\ra   
\Big) \dd t,
\end{align}
with $R \in L^\infty(0,T)$ given by
\be
\label{R}
\left\{ 
\begin{aligned}
R(t)&:=  q(\B) + \la Q(\Psih-\Psi_d),\B_2\B \ra + \la\ph(t),r(t) \ra,
\\
r(t) &:= \B_2^2 f(t)  -\A\B_2\B_1+2\B_2\A\B_1 
- 
\big[M_1,\B_2\big] \Psih.
\end{aligned}
\right. \ee

\begin{theorem}\label{thm:Q_Omega}
  For $v\in L^1(0,T)$ and $w \in AC(0,T)$
given by Goh transformation \eqref{GohT}, there holds 
  \be
  \Q(z[v],v)=\hatQ(\xi[w],w,w(T)).
  \ee
\end{theorem}
\begin{proof}
  For the contributions of the terms with 
$q(\cdot)$ and $q_T(\cdot)$,
we replace $z$ by $\xiz + w \B$.
For the contribution of the bilinear term in \eqref{tildeQ} we proceed
 as follows. There holds 
 \begin{equation}\label{step1}
\begin{aligned}
  \int_0^T v(t) \la \ph(t),\B_2 z(t)\ra\dd t & = \int_0^T v(t)w(t) \la \ph(t), \B_2 \B(t) \ra\dd t  + \int_0^T v(t) \la \ph(t), \B_2 \xiz(t)\ra \dd t\\
&=:l_1(w)+l_2(w).
  \end{aligned}
\end{equation}
There holds 
\begin{equation}\label{step2}
  \begin{aligned}
    l_1(w) &=\int_0^T v(t)w(t) \la \ph(t), \B_2 \B_1
  \ra + v(t)w(t) \la \ph(t), \B_2^2 \Psih(t)
  \ra\dd t \\ &=:g_1( w) + g_2(w).
\end{aligned}
\end{equation}
We apply several times Corollary \ref{lem:int_by_parts-1.c} 
for  $a:= \uh \B_2$ and 
(as can be checked in each case) $[B,a]=0$, 
and to begin with 
\be
y:=\B_1; \;\; B:=\B_2; \;\; b := \A\B_1 - \uh \B_2 \B_1.
\ee
By \eqref{hyp-Goh-tr1} and since $\uh\in L^\infty(0,T)$,
\eqref{yp}-\eqref{yp-bis} holds. We get:
\be
\begin{aligned}
  g_1(w) &=    \half w(T)^2 \la \ph(T),\B_2\B_{1}\ra
- \half\int_0^T w(t)^2 \la\ph(t), \B_2 (
  \A\B_1-\uh\B_2 \B_1) \ra\dd t \\ & \quad 
+\half\int_0^T w(t)^2 \la Q(\Psih(t) - \Psi_d(t)),\B_2\B_1 \ra
  \dd t  - \half\int_0^T
  w(t)^2 \la\ph(t),M_1 \B_1 \ra \dd t.
\end{aligned}
\ee
Applying (with similar arguments)
Corollary \ref{lem:int_by_parts-1.c} with
\be
y := \Psih; \;\; B := \B^2_2; \;\; 
b := f +\uh \B_1,
\ee
we get 
\be
\begin{aligned}
g_2(w)&= 
\half w(T)^2 \la \ph(T),\B_2^2 \Psih(T)\ra 
+ \half \int_0^T w(t)^2 
\la Q(\Psih(t) - \Psi_d(t)),\B_2^2 \Psih(t) \ra \dd t
\\
& \quad -\half \int_0^T w(t)^2 ( M_2^* \ph(t), \Psih(t) \ra\dd t 
 - \half\int_0^T w(t)^2 (\ph(t), \B_2^2 ( f(t)+\uh(t) \B_1)  \ra\dd t.
\end{aligned}
\ee
Finally setting 
\be
y := \xiz; \;\; B := \B_2; \;\; 
b := w b^1_z,
\ee
we get with Corollary \ref{lem:int_by_parts-1.c} 
with $b^1_z$ defined in \eqref{equ-bunz}:
\begin{equation}\label{step3}
  \begin{aligned}
    l_2(w)&=   w(T) \la\ph(T),\B_2 \xi_{zT} \ra + \int_0^T  w(t)
    \la Q(\Psih(t)-\Psi_d(t)),\B_2\xiz(t)\ra  \dd t
   \\ & \quad  - \int_0^T w(t)^2 \la\ph (t),\B_2 b^1_z(t)\ra \dd t
    - \int_0^T w(t)\la M^*_1 \ph(t), \xiz(t)\ra \dd t. 
\end{aligned}
\end{equation}
Combining the previous equalities, 
 the result follows.
\end{proof}

Given $\uh\in\calu_{ad}$, 
we write $PC_2(\uh)$ for the closure in the $L^2\times \RR$--topology  of the set
\be
\label{def-pc2}
PC(\uh) := \{(w,h) \in W^{1,\infty}(0,T) \times \RR, \dot w \in 
C(\uh); \; w(0)=0, \; w(T)=h  \}.
\ee
The final value of $w$ becomes an independent variable when we
consider this closure.

\begin{lemma}\label{lem3}
Let $\uh$ be a weak minimum for problem \eqref{problem:OC}.
Then 
  \be
\label{lem3-ome}
  \hatQ(\xi[w],w,h) \ge 0\quad \text{for all } (w,h) \in PC_2(\uh).
  \ee
\end{lemma}
\begin{proof}
  Let $(w,h) \in PC(\uh)$ with 
 $\dot w = v \in C(\uh)$. By Theorem \ref{thm:nec_sec_ord_cond},
  $\Q(z[v],v)\ge 0$, and so
  by Theorem \ref{thm:Q_Omega}, 
$0 \leq \Q(z[v],v) = \hatQ(\xi[w],w,w(T))$.
By \eqref{hyp-Goh-tr1},
$\hatQ(\xi,w,h)$
has a continuous extension to the space $W$ 
defined in \eqref{def--ww}.
The conclusion follows.
\end{proof}

\begin{definition}[Singular arc]
  A control $u \in \calu_{ad}$ is said to have a {\em singular arc} over
 $(t_1,t_2)$, with $0\leq t_1 < t_2 \leq T$,  if, 
for all $\theta\in (0,\half(t_2-t_1)),$ there exists $\eps>0$ such that
  \be
  u(t) \in
  [u_m+\varepsilon,u_M-\varepsilon],\quad \text{for a.a. } t\in (t_1+\theta,t_2-\theta).
  \ee
  We may also say that $(t_1,t_2)$ is a singular arc itself.
  We call $(t_1,t_2)$ a {\em lower boundary arc} if $u(t)=u_m$ for
  a.a. $t \in (t_1,t_2)$, and an {\em upper boundary arc} if $u(t)=u_M$
  for a.a. $t\in (t_1,t_2)$. 
We sometimes simply call them boundary arcs.
We say that a boundary arc $(c,d)$ is {\em initial} if $c=0$, and {\em final} if $d=T$. 
\end{definition}

\begin{corollary}
\label{cor2}
Let $\uh$ be a weak minimum for problem \eqref{problem:OC}. 
 Assume that 
\be
\Psi_d\in L^\infty(0,T,\calh),
\ee
and that 
\be
\label{comp-hyp-xi}
\text{the mapping $w\mapsto \xi[w]$ is compact from
$L^2(0,T)$ to  $L^2(0,T; \H)$.}
\ee
Let $(t_1,t_2)$ be a singular arc. Then
$R \in L^\infty(0,T;\H)$ defined in \eqref{R} satisfies
\be
\label{cor2-Rpos}
R(t) \ge 0 \quad \text{for a.a. }t\in (t_1,t_2).
\ee
\end{corollary}

\begin{proof}
Consider the set 
\be
P := \left\{ (w,h) \in PC_2(\uh); \quad w(t)=0 \;\;
\text{a.e. over $(0,T) \setminus (t_1,t_2)$.} \right\}.
\ee
By definition,  $P \subset   PC_2(\uh)$ and, therefore,
$$
\hatQ(\xi[w],w,h) \geq 0,\quad \text{for all } (w,h) \in P.
$$
Over $P$, $\hatQ$ is nonnegative (and therefore convex),
and continuous, and hence, is weakly l.s.c. 
By \eqref{comp-hyp-xi}, the terms of  $\hatQ$ 
where $\xi$ is involved are weakly continuous.
So, $\hatQ_b$  must be 
weakly l.s.c. over $P$. 
As it is well known, see e.g.  \cite[Theorem 3.2]{MR0046590}, 
this holds iff $R(t)\geq 0$ a.e. on $(t_1,t_2).$
The conclusion follows.
\end{proof}

\section{Second order sufficient optimality conditions}
\label{sec-suffici}

Given $\uh$ and $u$ in $\calu_{ad}$ with associated states 
$\Psih$ and $\Psi$ resp., 
setting $v:=u-\uh$ and $z:=z[v]$, 
we recall that $\delta \Psi:= \Psi-\hat\Psi$ and
$\eta := \delta\Psi-z$ are solution of 
\eqref{semdeltaPsi} and \eqref{sem-eta}, resp.

\subsection{Goh transform for $\delta\Psi$}
\label{secgtdelp}
We apply Lemma \ref{eq-smgs} to
\eqref{semdeltaPsi}, with here
$(a_{\delta\Psi} ,b_{\delta\Psi} ) = (u \B_2,\calb)$.  
Using again \eqref{dphipaphip}
we obtain by the same arguments that 
the dynamics for $\xi_{\delta \Psi}:=\delta\Psi-w\B$
reads
\be\label{eq:xi_delta_Psi}
\dot \xi_{\delta \Psi} + \A \xi_{\delta \Psi}  = u \B_2 \xi_{\delta \Psi} + w b^1_{\delta\Psi} ; 
\ee
Since $g_{\delta\Psi} = g_z$, we have by \eqref{hyp-Goh-tr1}
that the amount below belongs to $L^2(0,T;\H)$:
\be
\begin{aligned} 
b^1_{\delta\Psi}  
 = u \B_2 \calb  - g_{\delta\Psi}
= v \B_2 \calb  - \B_2 f - M_1 \Psih - \A \B_1.
\end{aligned}
\ee

\begin{corollary}\label{cor:dPsi}
We have that 
\be\label{est:dPsi1}
\| \xi_{\dPsi}\|_{C(0,T;\H)} =O(\norm{w}{2}),
\ee
\be
\label{est_dPsi}
\| \delta\Psi\|_{L^2(0,T;\H)} =O(\norm{w}{2}).
\ee
\end{corollary}
\begin{proof}
Consequence of Lemma \ref{eq-smgs} 
with here $b^0=\B$.
\end{proof}

\subsection{Goh transform for $\eta$}
\label{secgteta}
We next apply Lemma \ref{eq-smgs} to the equation \eqref{sem-eta},
with now 
$(a_\eta,b_\eta) = (\uh \B_2, \B_2 \delta\Psi)$.
We need to apply Corollary \ref{lem:int_by_parts-1.c}
with $B:=\B_2$ and $y:= \delta\Psi$.
Similarly to \eqref{phi-zz}
we obtain that $\Phi:=By$ satisfies
\be
\label{phi-zzeta}
\dot \Phi+ \A \Phi
= g_\eta, \quad
\text{with } g_\eta :=
u \B^2_2 \delta\Psi  + v \B_2\B
+ M_1  \delta\Psi.
\ee
The dynamics for $\xi_\eta:=\eta-w \B_2 \delta\Psi$
reads
\be \label{equ:xi_eta}
\dot\xi_\eta + \A \xi_\eta  = \uh \B_2 \xi_\eta + w b^1_{\eta} ; 
\ee
with 
\be
\label{buneta-exp}
\begin{aligned} 
b^1_{\eta}  
=
a_\eta b_\eta  - g_\eta
= -v \B^2_2 \delta\Psi - v \B_2 \B -
M_1   \delta\Psi.
\end{aligned}
\ee

\begin{lemma}\label{est:eta}
We have that 
\be
\label{est4}
\|\eta\|_{L^\infty(0,T;\H)} =  O\big(\|v\|_2  \|w\|_2 \big).
\ee
\end{lemma}

\begin{proof}
By Corollary \ref{cor:dPsi}, we have that 
\be
\| v \delta\Psi \|_{L^1(0,T;\H)}
\leq
\|v\|_2 \| \delta\Psi\|_{L^2(0,T;\H)} =O(\|v\|_2\norm{w}{2}).
\ee
We conclude with the first equality in \eqref{esteta}.
\end{proof}

\subsection{Main results}\label{sec:soonoc-s}

In this section we state a sufficient optimality condition, 
that needs a new notion of optimality. The
control $\uh$ is  said to be a {\em Pontryagin minimum} 
(see e.g. \cite{MR1641590})
for  problem \eqref{problem:OC} if there exists $\eps>0$ such that $\uh$ is optimal among all the controls $u\in \U_{\rm ad}$ verifying
$
\|u-\uh\|_1<\eps.
$
A bounded sequence $(v_k)\subset L^\infty(0,T)$ is said to
converge to 0 {\em in the Pontryagin sense} 
if $\norm{v_k}{1} \rightarrow 0$.

We need some additional hypotheses: 
\be
\label{hyp-Goh-tr2}
\left\{ 
\ba{lll}
{\rm (i)} & 
\B^2_2 f\in C(0,T;\calh);
\; \; \Psi_d \in C(0,T;\calh), 
\\[1.5ex] {\rm (ii)} & 
M^*_k \ph \in C(0,T;\H^*), \;\; k=1,2.
\ea\right.
\ee 

The following result states a refinement of the Taylor expansion 
stated in Corollary \ref{cor:expansion2}.

\begin{proposition}\label{prop:expansion_w}
Let $\uh\in\calu_{ad}$ 
and let $(v_k)$ converge to $0$ in the Pontryagin sense.
Then 
\be 
J(\uh+v_k)=J(\uh)+\int_0^T \Lambda(t) v_k(t) dt + 
\half \hatQ(\xi[w_k],w_k,w_k(T))+o(\norm{w_k}{2}^2+w_k(T)^2),
\ee
where $(\xi[w_k],w_k)$ is obtained by the Goh transform.
\end{proposition}

\begin{proof} 
First observe that, in view of the definitions of 
$\xi$ and $\xi_{\delta \Psi}$,
and of \eqref{est1} and \eqref{est:dPsi1}, we have that
\be
\label{z-fin-est}
\|z(T)\|_\H \leq 
\|\xi(T)\|_\H 
+ |h| \|\B\|_{L^\infty(0,T;\H)}
= O( \|w\|_2 + |h|),
\ee
and
\be
\label{deltapsifinalest}
\|\delta \Psi(T)\|_\H \leq 
\|\xi_{\delta \Psi}(T)\|_\H 
+ |h| \|\B\|_{L^\infty(0,T;\H)}
= O( \|w\|_2 + |h|).
\ee

We skip indexes $k$. 
Recalling that $u=\uh+v$, $\dPsi$ and $z$ are the solutions of \eqref{semdeltaPsi} and \eqref{semg5lin}, respectively, and $\eta= \dPsi - z,$ there holds the identity
  \be\label{expGoheq1}
  \begin{aligned}
  Q(\dPsi,v) - Q(z,v) &=\int_0^T \la Q(\dPsi(t) + z(t)),\eta(t)\ra\dd t 
  + \la Q_T(\dPsi(T) +  z(T)),\eta(T)\ra  \\
  &\quad + 2\int_0^T v(t) \la\ph(t),\B_2 \eta(t)\ra \dd t.
  \end{aligned}
  \ee
By \eqref{z-fin-est}-\eqref{deltapsifinalest},
Corollary \ref{cor:dPsi} and Lemma \ref{est:eta}, 
 the first and second terms of the r.h.s. are of order
  $o(\norm{w}{2}^2+h^2).$
  Recall now \eqref{sem-eta}, and set
\be
y:= \eta, \quad  a:=\uh \B_2,\quad b:= v\B_2\delta\Psi,\quad B:=\B_2.
\ee
Using Corollary \ref{lem:int_by_parts-1.c}
(in fact, several times in the proof), 
the last integral  in \eqref{expGoheq1} can be rewritten as
  \be
\label{eq-kdil}
  \begin{aligned}
    & \int_0^T v(t) \la \ph(t),\B_2 \eta(t) \ra \dd t = [w \la \ph,\B_2 \eta \ra
   ]_0^T  +  \int_0^T  w(t) 
 \la Q(\Psih(t) -\Psi_d(t)) , \B_2 \eta \ra  \dd t \\
 & - \int_0^T  w(t) 
 v(t) \la\ph(t), \B_2^2 \dPsi(t)\ra  \dd t 
- \int_0^T w(t) \la M^*_1 \ph(t), \eta(t)\ra \dd t. 
\end{aligned}
\ee
By arguments already used, 
all terms of the r.h.s. of \eqref{eq-kdil}
are of order $o(\norm{w}{2}^2+h^2)$, 
except maybe for the third term. 
Recall the equation \eqref{semdeltaPsi} for $\delta\Psi$ and define
\be
a:=\uh \B_2 ,\quad b:=  v\B_2\delta\Psi + v\B,\quad B:=\B^2_2,
\ee 
and we have:
\be
\begin{aligned}
 \int_0^T   &w(t) 
 v(t) \la \ph(t),\B^2_2 \dPsi(t)\ra  \dd t   \\
 =& \,\half [w^2 \la \ph,\B^2_2 \dPsi \ra
   ]_0^T    +\half \int_0^T  w(t)^2 
 \la Q  (\Psih(t) - \Psi_d(t)), \B^2_2 \dPsi(t)\ra \dd t 
\\ & 
 -\half \int_0^T  w(t)^2  \Big\la\ph(t), v(t) \B_2^2
 \big(\B_2\delta\Psi+\B(t)\big) \Big\ra \dd t
\\ & 
 -\half \int_0^T  w(t)^2  \Big\la M_2^* \ph(t),  \dPsi(t)\Big\ra \dd t
\end{aligned}
\ee
Here, again, by the same arguments,
using that $v$ is uniformly essentially bounded and \eqref{est_dPsi}
we find that all terms are of order $o(\norm{w}{2}^2+h^2)$, 
except maybe for the integral $\int_0^T  w(t)^2 
 v(t)\la \ph(t),\B_2^2\B(t))\ra  \dd t, $ which can be 
integrated using Corollary \ref{lem:int_by_parts-1.c} for
 \be
y := \B, \quad  a:= \uh \B_2 ,\quad b:= \A\B_1+M_1 \hat\Psi+\B_2f,\quad B:= \B_2^2.
 \ee
Hence we get
\be \label{wvpB}
\begin{aligned} 
 \int_0^T  w(t)^2 
 v(t)&\la \ph(t),\B_2^2\B(t))\ra  \dd t  
  =
 \frac{1}{3} [w^3 \la \ph,\B_2^2 \B \ra]_0^T    \\
 &+\frac{1}{3} \int_0^T  w(t)^3 
\la Q(\Psi(t) - \Psi_d(t)), \B_2^2 \B \ra \dd t  
\\ & 
 - \frac{1}{3} \int_0^Tw(t)^3\la\ph(t), 
\B^2_2 ( \A \B_1 + M_1 \Psih + \B_2 f)  \ra \dd t
\\ & -
\frac{1}{3} \int_0^Tw(t)^3\la M^*_2\ph(t), \B \ra \dd t,
 \end{aligned}
\ee
The first term in the right-hand side of \eqref{wvpB} is of order
$o(h^2),$ 
while the other three have the form $\int_0^T w^3(t) q(t) \dd t$ for 
$q\in  L^\infty(0,T)$. Note in particular that 
\be
\la \ph(t),M_1 \Psih(t) \ra_\H
=
\la M_1^* \ph(t), \Psih(t) \ra_{E_1}
=
\la M_1^* \ph(t), \Psih(t) \ra_{\H}
\ee 
combined with \eqref{hyp-Goh-tr2}(ii)
implies that the above product 
is essentially bounded.
Then the following estimate holds
\be
\left| \int_0^T w(t)^3 q(t) \dd t \right|
\leq 
\|w\|_\infty \|w\|^2_2 \|q\|_\infty = o( \|w\|^2_2),
\ee
we get
\be
Q(\dPsi,v) - Q(z,v) = o(\norm{w}{2}^2+h^2).
\ee
Finally, with Proposition \ref{ExpEasy} and Theorem \ref{diffcostfct} the result follows.
\end{proof}

Remember that $\Lambda$ was defined in \eqref{expr-lambda}.
In the following we assume that the following hypotheses hold: 
\begin{enumerate}
\item {\em finite structure:}
\be\label{finite_structures}
\left\{
\begin{array}{l}
\text{there are finitely many boundary and singular maximal arcs} \\ \text{and the closure of their union is $[0,T],$}
\end{array}
\right.
\ee
\item {\em strict complementarity} for the control constraint
(note that $\Lambda$ is a continuous function of time)
\be
\label{strict_complementarity}
\left\{ \begin{array}{l}
\text{$\Lambda$ has nonzero values over the interior of each boundary
  arc, and}
\\
\text{at time 0 (resp. $T$) if an initial (resp. final) boundary
  arc exists,}
\end{array}\right.
\ee
\item letting $\T_{BB}$ denote the {\em set of bang-bang junctions,} we assume
\be
\label{RposTBB}
R(t) >0,\quad t\in \T_{BB}.
\ee
\end{enumerate}

\begin{proposition}
Let $\uh \in \calu_{ad}$ satisfy 
\eqref{finite_structures}--\eqref{strict_complementarity}.
Then $PC_2(\uh)$, that was defined before \eqref{def-pc2},
satisfies
\be
PC_2(\uh)=
\left\{
\begin{array}{l}
(w,h)\in L^2(0,T)\times \cR;\text{ $w$ is constant over boundary arcs,}\\ \text{ $w=0$ over an initial boundary arc }\\
\text{and $w= h$ over a terminal boundary arc}
\end{array}
\right\}.
\ee
\end{proposition}

\begin{proof}  
Similar to the one of \cite[Lemma 8.1]{ABDL12}. 
\end{proof}

Consider the following positivity condition:
there exists $\alpha>0$ such that
\be
\label{sufcondso}
\hatQ(\xi[w],w,h) \geq \alpha ( \|w\|^2_2 + h^2),
\quad\text{for all $(w,h) \in PC_2(\uh)$.}
\ee
We say that $\uh$ satisfies a {\em weak quadratic growth condition}
if there exists $\beta>0$ such that for any  $u \in \U_{ad},$
setting $v:=u-\uh$ and $w(t):=\int_0^Tv(s)\dd s,$ we have
\be
\label{sufcondqg}
F(u) \geq F(\uh) + \beta ( \|w\|^2_2 + w(T)^2),
\quad\text{if $\|v\|_1$ is small enough.}
\ee
The word `weak' makes reference to the fact that the growth is
obtained for the $L^2$ norm of $w$, and not the one of $v$.

\begin{theorem}
\label{thm-sufcond}
Let $\uh$ be a weak minimum for
problem \eqref{problem:OC},
satisfying 
\eqref{finite_structures}-\eqref{RposTBB}. 
Then \eqref{sufcondso} holds iff  the quadratic growth condition
\eqref{sufcondqg}  is satisfied.
\end{theorem}

\begin{proof}
Let \eqref{sufcondso} hold and let $(v_k,w_k)$ contradict the weak quadratic growth condition \eqref{sufcondqg},
 i.e. 
 \begin{align}
 \uh + v_k \in \U_{ad}, \quad v_k \neq 0, \quad \norm{v_k}{L^1(0,T)} \rightarrow 0,\quad w_k(t)=\int_0^t
 v_k(s)ds,
 \end{align}
 with
  \begin{align}\label{quadratic_growth_condition}
    J(\uh + v_k) \le J(\uh) + o(\gamma_k)
\end{align}
for $\gamma_k := \gamma(w_k,w_{k,T})$ where $\gamma(w,h):= \norm{w}{2}^2+ h^2,$ for any $(w,h) \in L^2(0,T) \times \cR.$
Set $h_k:=w_{k,T},$ and 
$(\wh_k,\hh_k):=(w_k,h_k)/\sqrt{\gamma_k}$
that has unit norm in
$L^2(0,T)\times \RR$. 
Extracting if necessary a subsequence, 
we have that there exists 
$(\wh,\hh)$ in $L^2(0,T) \times \cR,$ 
such that $\wh_k$ converges weakly in $L^2(0,T)$ to  $\wh$ 
and $\hh_k \rightarrow \hh.$
 Let $\hat\xi_k$ and $\hat\xi$  denote the solution of
 \eqref{diffeqxi} associated 
with $\wh_k$ and $\wh,$ respectively. 
Since $w\mapsto \xi[w]$ is linear and continuous 
$L^2(0,T) \rar L^\infty(0,T;\H)$,
$\hat \xi_k$ weakly converges to $\hat \xi$ in $L^\infty(0,T;\H).$
By the compactness hypothesis 
\eqref{compachyp}
we also have that 
$\hat \xi_k\rar\hat \xi$ in $L^2(0,T;\H).$

We proceed in three steps, starting by proving the sufficiency of 
\eqref{sufcondso}.
We obtain in {\em Step 1} that
$(\wh,\hh) \in PC_2(\uh)$, and in {\em Step 2}
that $(\wh_k,\hh_k)\rar 0$ strongly in $L^1(0,T)\times \cR,$ which contradicts the
    fact that $(\wh_k,\hh_k)$ has unit norm. 
Finally in  {\em Step 3} we prove the necessity of \eqref{sufcondso}.
   
{\em Step 1.} From Proposition \ref{prop:expansion_w} we have
    \begin{align}\label{exp1}
J(\uh+v)=J(\uh)+\int_0^T \Lambda(t) v(t) dt + O(\gamma_k).
    \end{align}
    Note that the integrand on the right hand-side of the previous equation is nonnegative in view of the first order conditions given Proposition \ref{Prop1order}.
  Using     \eqref{quadratic_growth_condition}, it follows that
    \be
\label{gamkrar0}
 \lim_{k\rightarrow \infty} \frac 1 {\sqrt \gamma_k}\int_0^T \Lambda(t) v_k(t) \dd t = 0.
    \ee
Consider now a maximal boundary arc $[c,d]$ and let $\eps>0$ be sufficiently small such that $c+\eps < d-\eps.$
In view of hypotheses \eqref{strict_complementarity}, $\Lambda$ is uniformly positive (respectively, uniformly negative)
on $[c+\eps,d-\eps]$, and therefore, from \eqref{gamkrar0} we get
\be
\label{deltawk}
0 = \lim_{k\rightarrow \infty} 
\frac 1 {\sqrt \gamma_k}\int_{c+\eps}^{d-\eps} v_k(t) \dd t = \lim_{k\rightarrow \infty}
\wh_k(d-\eps) - \wh_k(c+\eps).
\ee
Since $\wh_k$ is monotonous on $[c,d]$ and $\eps>0$ is arbitrarily
small, 
it follows that, extracting if necessary a subsequence,
we can assume that 
$\wh_k$  converges uniformly on $[c+\eps,d-\eps]$ to a constant
function. By a diagonal argument we may assume that  
$\wh$ is constant on every of (the finitely many) 
boundary arcs $[c,d]$.

For an initial (resp. final) boundary arc, 
in view of the strict complementarity hypothesis
\eqref{strict_complementarity}
we have a similar argument using integrals
between 0 and $d-\eps$
(resp. between $c+\eps$ and $T$). 
Since $\wh_k(0) =0$ (resp. $\wh_k(T)= \hh_k$), 
we deduce that, on this arc, $\wh$ equals 0
(resp. $h$). Hence, we showed that $(\wh,\hh) \in PC_2$ as desired.

{\em Step 2.}  From \eqref{sufcondqg}, Proposition
\ref{prop:expansion_w}, 
the non-negativity of $\int_0^T
\Lambda(t)v_k(t)dt $, 
\eqref{quadratic_growth_condition},
and the convergence of
$\xi_k$ to $\hat \xi$ in $L^2(0,T;\calh)$ 
we deduce that 
  \begin{align}\label{cond_o}
   \hatQ(\xih_k, \wh_k,\hw_{T,k}) \le o(1).
  \end{align}
 
  Let us consider the set $I_S := [0, T ] \setminus (I_m \cup I_M )$ the closure of the union of singular arcs, and recall the definition of
  $\T_{BB}$
in \eqref{strict_complementarity}.
 We set for $\eps > 0$
  \be
  I^{\eps}_{SBB}:= \{t \in [0, T ];
  \,\, \dist(t, I_{S} \cup \T_{BB}) \le  \eps \},\quad
  I^\eps_0 := [0, T ] \setminus I^\eps_{SBB}.
  \ee
  Recalling that $w_k$ converges uniformly on $[c+\eps,d-\eps]$ for any bang arc $[c,d]$ and $\eps>0$ sufficiently small, we deduce that $w_k$ convergence uniformly on the set $I^{\eps}_0$.

Recall the definition of $R$ in \eqref{R}.
Observe that $R(t)$ is continuous in view of the continuity of
$f(t)$ and $\Psi_d(t)$ in $\H$, and of $\psih(t)$ in $E$. 
By \eqref{sufcondso}, there exists $\alpha >0$, such that the quadratic form  $\hatQ(\xi[w],w,h) - \alpha \gamma(w,h)$
is nonnegative over $PC_2(\uh)$.
So, by hypothesis \eqref{RposTBB} and Corollary \ref{cor2},
we have that 
  \be \label{R_on_SBB}
  R(t)\ge \half \alpha \text{ over } I^{\varepsilon}_{SBB}.
  \ee
  We split the form $\hatQ$ defined in \eqref{Omega} as
$   \hatQ=  \hatQ_{T,a} +  \hatQ_b^1+ \hatQ_b^2,$
  where
  \be
  \begin{aligned}
     \hatQ_{T,a}&:= \hatQ_T + \hatQ_a,\quad
        \hatQ_b^1(w):= \int_{I_{SBB}^{\varepsilon}} R(t) w(t)^2 \dd t,\quad
     \hatQ_b^2(w):= \int_{I^{\varepsilon}_0} R(t) w(t)^2 \dd t.
  \end{aligned}
  \ee
By \eqref{comp-hyp-xi},
$\hatQ_{T,a}(\xi[\cdot],\cdot, h) \colon L^2(0,T) \times \cR
\rightarrow  \cR$
is weakly continuous.
By \eqref{R_on_SBB}, the restriction of 
$ \hatQ_b^1$ to $L^2(I^{\varepsilon}_{SBB} )$ is a
  {\em Legendre form}
(it is weakly l.s.c. and, if $w_k$ weakly converges to $\hat w_k$ and 
 $ \hatQ_b^1( w_k) \rightarrow
    \hatQ_b^1( w_k),$ then $ w_k \rightarrow \hat\omega$ strongly in $L^2(I^{\varepsilon}_{SBB})$).
  Thus we have
\be  \begin{aligned}\label{limOmega}
    \hatQ_{T,a}(\xih,\wh,\hh) &= \lim_{k} \hatQ_{T,a}(\xih_k,\wh_k,\wh_{k,T}),\\
    \hatQ_b^1(\wh) &\le \liminf_{k \rightarrow \infty} \hatQ_b^1(\wh_k),\\
    \hatQ_b^2(\wh)&=\lim_{k\rightarrow \infty}\hatQ_b^2(\wh_k).
  \end{aligned} \ee
The last equality uses the fact that $\wh_k\rar\wh$ uniformly on $I_0^{\eps}$.
From \eqref{sufcondso}, \eqref{limOmega} and  \eqref{cond_o}
and step 1, we get:
\be
\begin{aligned}
\alpha\gamma(\wh,\hh) & \leq \hatQ(\xih,\wh,\hh) \leq \lim_{k\rightarrow \infty} \hatQ_{T,a}(\xih_k,\wh_k,\wh_{k,T})
+ \limsup_{k \rightarrow \infty} \hatQ_b^1(\wh_k) + \lim_{k\rightarrow \infty}\hatQ_b^2(\wh_k) \\
&= \limsup_{k \rightarrow \infty} \hatQ(\xih_k,\wh_k,\hh_k)  \leq 0.
\end{aligned}
\ee
Then, $(\wh,\hh)=0$ and
$\hatQ_b^1(\wh_k) \rightarrow \hatQ_b^1(\wh)=0.$ Since 
$\hatQ_b^1$ is a Legendre form,
$\wh_k \rightarrow \wh=0$ in $L^2(I^{\varepsilon}_{SBB})$. Given that
$\wh_k$ converges uniformly to $\wh$ on $I_0^\eps,$ we get that
$(\wh_k,\hh_k)$ 
strongly converges to $(0,0)$ on $L^2(0,T)\times \cR$.
This leads to a contradiction since $(\wh_k,\hh_k)$  is a unit sequence. Thus, the quadratic growth \eqref{sufcondqg} holds.

{\em Step 3.}  
Conversely, let the weak quadratic growth condition \eqref{sufcondqg} be given for $\beta>0.$ Further let $v\in L^2(0,T)$ and $w[v](s):=\int_0^Tv(s)\dd s$. Applying the second order necessary condition (see Lemma \ref{lem3}) to problem 
\be
\min J(u,\Psi) - \beta \gamh, \quad \gamh:=\int_0^T w[v](s)^2 \dd s + w[v](T)^2
\ee
we obtain condition \eqref{sufcondso}.
\end{proof} 

\section{Applications}\label{sec:applschr}

In this application section, after a general discussion for the case
of diagonalizable operators, where the semigroup properties can be
related to the structure of the spectrum, we consider two important
application fields, the heat and wave equations. It is of interest to
see the great qualitative difference between them, related in
particular
to the fact that for the wave equation, the commutators
involve no differential operators.

\subsection{Diagonalizable operators}
\label{sec:notation-sch-d}
In our applications $\H$ is a separable  Hilbert space with a Hilbert basis $\{e_k;\, k\in\NN\},$ of eigenvectors of 
$\A$, with associated (real) eigenvalues $\mu_k$. 
Let $\Psi\in \H$, with components $\Psi_k:= (\Psi,e_k)_\H,$ 
where $(\cdot,\cdot)_\H$ denotes the scalar product in $\H.$

We have that 
\be
\label{sec:notation-sch-d-h2}
\dom(\cala) =\left\{ \Psi\in \H; \;\; 
\sum_{k\in\NN} |\mu_k|^2 |\Psi_k|^2 < \infty \right\}. 
\ee
Given an initial condition $\Psi_0 = \sum_{k\in\NN} \Psi_{0k} e_k \in \H$,
the semigroup verifies the following expression:
\be
e^{-t\A}\Psi_0 = \sum_{k\in\NN} e^{-t\mu_k} \Psi_{0k} e_k.
\ee
Since $\Psi_0\in \H$ we have that $\sum_{k\in\NN} \|\Psi_{0k} \|_\H^2 < \infty.$
Let us note that the eigenvalues $\mu_k$ have to comply with
condition \eqref{cond:bounded}, 
i.e. 
\be
\label{cond:bounded2-muk}
\sum_{k\in\NN} |e^{-t\mu_k}|^2  | \Psi_k(t) |^2
\leq 
\left( c_{\A} e^{\lambda_{\A} t} \right)^2
\sum_{k\in\NN} | \Psi_k(t) |^2.
\ee
Letting ${\rm Re}$ denote the {\em real part,} we observe that 
$|e^{-t\mu_k}|  = e^{-t \rm Re (\mu_k)}$, 
so that the above condition \eqref{cond:bounded2-muk} is equivalent to
\be
\label{cond:bounded2-muk-bis}
\sum_{k\in\NN} e^{-2t \rm Re (\mu_k)} | \Psi_k(t) |^2
\leq 
\left( c_{\A} e^{\lambda_{\A} t} \right)^2
\sum_{k\in\NN} |\Psi_k(t)|^2.
\ee
Considering the case when $\Psi_0=e_k,$ for some $k\in\NN$, 
we observe that 
\eqref{cond:bounded2-muk-bis}
holds iff the following
{\em bounded deterioration} condition holds:
\be \label{cond:bounded_deter}
\gamma := \inf_k  \mu_k > - \infty.
\ee
Then \eqref{cond:bounded2-muk-bis} holds with 
$c_{\A} =1$ and $\lambda_{\A} =\gamma$, 
and consequently:
\be\label{cond:bounded-ter}
\| e^{-t \A}\|_{\L(\calh)} \leq e^{-\gamma t},\quad t>0.
\ee
Observe that, if $\gamma\geq 0$, then the semigroup results a contraction semigroup.

In this setting we have the regularity results that follow. 
Set, for $q>0,$
\be
\H^q :=\{\Psi\in \H; \;\; \sum_{k\in\NN}  
(1+| \mu_k|^q) |\Psi_k|^2 < \infty\},
\ee
(so that $\H^2 = \dom(\A)$), 
endowed with the norm
\be
\| \Psi \|_{\H^q} := \left(\sum_{k\in\NN} 
(1+ |\mu_k|^q) |\Psi_k|^2 \right)^{1/2}.
\ee
Then $\H^q$ is a Banach space with dense, continuous 
inclusion in $\H$.
Since, for $0<q <p$ and $a>0,$ it holds
$a^q \leq 1 + a^p$,
we have that 
$\H^q \subset \H^p$. 
Furthermore, under the bounded deterioration condition \eqref{cond:bounded_deter},
it holds $e^{-t\A}(\H^q) \subseteq \H^q$ and the restriction of $e^{-t\A}$ to $\H^q$ is itself a semigroup.

\begin{remark} 
By the Hille-Yosida Theorem, $\A$
is the generator of a semigroup iff,
for some $M>0$ and $\om\in \RR$,
for all $\lambda > \om$,
and $n=1,2,\ldots,$
$(\lambda I + \A)$ has a continuous inverse that satisfies
\be \label{Hille_Yosida}
\| (\lambda I + \A) ^{-n} \|_{\L(\H)} \leq M/(\lambda-\om)^{n}.
\ee 
That is, $\lambda+\mu_k\neq 0$ for all $k$, and 
for all $f = \sum_k f_k e_k \in\H$, 
\be
\sum_k |\lambda+\mu_k|^{-2n}|f_k|^2
\leq
{M^2}(\lambda-\om)^{-2n} \sum_k |f_k|^2.
\ee
This holds iff, for all $k$, 
$|\lambda+\mu_k|^{-2n} \leq {M^2}/(\lambda-\om)^{2n}$,
that is, 
\be
\label{equ-cinq-d}
|\lambda+\mu_k| \geq (\lambda-\om)/M^{1/n}.
\ee
Now, consider $M=1$ and  note that \eqref{equ-cinq-d} is equivalent to 
\be
2 \lambda (\om +\mu_k)
\geq 
|\mu_k|^2 +\om^2.
\ee
Dividing by $\lambda$ and taking $\lambda$ to $\infty$, we get
$\om +\mu_k \geq 0$.
As expected, we recover the bounded deterioration condition
\eqref{cond:bounded_deter} with $\om=-\gamma $, and we conclude that, with these choices of $M$ and $\om,$ the Hille-Yosida condition holds.
\end{remark}

In this setting we have some compact inclusions.

\begin{lemma}
\label{lem:comp-incl}
Let $0<q<p$.  Then the  inclusion of 
$\H^p$ into $\H^q$ is compact iff 
$|\mu_k| \rar \infty$. 
  \end{lemma}

\begin{proof}
  Part 1. Let $|\mu_k| \rar \infty$.  Reordering if necessary, we may
  assume that $|\mu_k|$ is a nondecreasing sequence.  Let $(\Psi^n)$ be
  a bounded sequence in $\H^p$.  Consider the truncation at order $N$,
   say $\varphi^{N,n} := \sum_{k< N} \Psi_k^ne_k \in \H^p.$  
The order $N$ can be taken large enough, so that 
$|\mu_N| > 1$. 
It is easily checked that 
\be
 \frac{1+ |\mu_k|^q }{1+ |\mu_k|^p} \leq
 \frac{1+ |\mu_N|^q }{1+ |\mu_N|^p},
\quad \text{for any $k>N$.} 
\ee
Then 
\be
\label{Psinphin}
  \ba{lll} 
\| \Psi^n - \varphi^{N,n}\|^2_{\H^q} 
&= \sum_{k\geq N} (1+   |\mu_k|^q) |\Psi^n_k|^2 
\\ & \leq \disp 
\frac{1+ |\mu_N|^q }{1+    |\mu_N|^p} 
\sum_{k\geq N}  (1+ |\mu_k|^p) |\Psi^n_k|^2 
\\ & \leq \disp
  \frac{1+ |\mu_N|^q }{1+ |\mu_N|^p} \| \Psi^n\|^2_{\H^p}.
\ea\ee
By a diagonal argument we may assume that
$\{\varphi^{N,n}\}_{n\in\NN}$ has, for every $N$,  a limit say
$w^N$ in $\H^q.$ 
By \eqref{Psinphin}, 
for any $\eps>0$, we can choose $N$ large enough such that 
$\|w^N - \Psi^n \|_{\H^q} \leq \eps.$  It follows that $\Psi^n$ is a 
Cauchy sequence in $\H^q$. 
\\ Part 2.
If there exists a subsequence $(k_j)\subset \cN,$
such that $\mu_{k_j}$
is bounded, then $(e_{k_j})$
is necessarily a bounded sequence in 
$\H^p$ (and therefore in $\H^q$)
that converges to zero weakly, but not strongly,
so that the inclusion  of 
$\H^p$ into $\H^q$ cannot be compact.
\end{proof}

\begin{lemma}
\label{lem:higher_reg}
If, for some $q>0$:
\be
\B_1\in \H^q; \;\; \B_2\in \L(\H^q); \;\; f \in L^1(0,T;\H^q);
\;\; \Psi_0 \in \H^q,
\ee
then the solution of \eqref{semg1} belongs to $C(0,T;\H^q)$. 
  \end{lemma}

\begin{proof}
Consequence of Lemma \ref{lem-reg.l1}
concerning the restriction property.
\end{proof}



\subsection{Link with the variational setting for parabolic equations}\label{sec:regsol-p}
The variational setting is as follows. 
Assuming as before $\H$ to be a Hilbert space, 
let $V$ be another Hilbert
space continuously embedded in $\H$, with dense and compact inclusion.
We identify $\H$ with its dual and therefore, by the Gelfand triple
theory, with a dense subspace of $V^*$. Given a continuous
bilinear form $a: V\times V \rar \RR$, we consider the equation
\be
\label{eq-vf1}
\la \dot \Psi(t) , v \ra_V + a( \Psi(t), v) = 
( f(t), v )_\H, \quad \text{for a.a. $t\in (0,T)$}
\ee
with $f\in L^2(0,T;\H)$ and the initial condition
$\Psi(0) = \Psi_0 \in \H$. 
It is assumed that the bilinear form
is semicoercive, that is, 
for some $\alpha>0$ and $\beta\in\RR$:
\be\label{semicoercive}
a(y,y) \geq \alpha \|y\|^2_V - \beta \|y\|^2_\H,
\quad \text{for all $y\in V$}.
\ee
By the Lions-Magenes theory \cite{LioMag68a},
equation \eqref{eq-vf1}
has a unique solution in the space
\be
W(0,T) := \{ u\in L^2(0,T,V); \;\; 
\dot u\in L^2(0,T,V^*) \}.
\ee
It is known that 
$W(0,T) \subset C(0,T;\H)$, so that 
$W(0,T) \subset L^2(0,T;\H)$. 
By Aubin's Lemma \cite{MR0152860},  
\be
\label{Aubin}
\text{the inclusion $W(0,T) \subset L^2(0,T;\H)$
is compact}.
\ee

Let $A_V \in \L(V,V^*)$ be defined by 
\be
\la A_V u,v\ra = a(u,v), \quad
\text{for all $u$, $v$ in $V$.}
\ee
The adjoint $A*_V \in \L(V,V^*)$ satisfies
\be
\la A^*_V u,v\ra = a(v,u), \quad
\text{for all $u$, $v$ in $V$.}
\ee
Since $V\subset \H$ we can consider the 
following unbounded operators
$\cala_\H$  and $\cala^*_\H$ in $\H$, with domain
\be
\dom( \cala_\H ):= \{ v\in V; \quad A_V v \in \H \};
\quad 
\dom( \cala^*_\H ):= \{ v\in V; \quad A^*_V v \in \H \},
\ee
and 
$\cala_\H v := A_V v$ for all $v \in \dom( \cala_\H )$,
$\cala^*_\H v := A^*_V v$ for all $v \in \dom( \cala^*_\H )$.
Then one can check that 
$\cala^*_\H$ is the adjoint of $\cala_\H$.

\begin{lemma}
\label{lem:general_parab_setting}
In the above setting, 
$\cala_\H$ is the generator of a semigroup, 
and when $f\in L^2(0,T;\H)$
the variational solution coincides with the mild solution. 
\end{lemma}

\begin{proof}
We first check that $\A_\H$ is the generator of a
semigroup thanks to the Hille-Yosida Theorem.
Let $\beta$ be given by the semicoercivity condition \eqref{semicoercive}. 
Set $a_\beta(y,z):= a(y,z) + \beta (y,z)_\H$. 
Let $f\in \H$. By the Lax-Milgram Theorem, there exists a
unique $y\in V$ such that
\be
a(y,z) = (f,z)_\H, \quad
\text{for all $v\in V$},
\ee
and in addition 
\be
| \la A_V y,z\ra | = | a(y,z) | = | (f,z)_\H | \leq \|f\|_\H \|z\|_\H 
\ee
proving that $A_V y\in \H$, and therefore $y\in\dom(\A_\H)$.
Also, 
\be
( \A_\H y,z)_\H = \la A_V y,z\ra_V = (f,z)_\H,
\ee
for any $z\in V$ (and therefore for any $z\in \H$),
means that $\A_\H y =f$.

In order to end the proof, in view of Theorem \ref{weaksense-thm}, it suffices to prove that 
weak and variational solutions coincide.
We only need to check that the strong
formulation implies the weak one. 
Taking $v=\psi\varphi$ in \eqref {eq-vf1},
with $\psi\in \cald(0,T)$ and $\varphi\in \dom(A^*_V)$
we get
\be
\int_0^T \psi(t) 
\left[ \la \dot \Psi(t) , \varphi \ra 
+ a( \Psi(t), \varphi) - \la f(t), \varphi\ra \right] \dd t
=0.
\ee
Since $\psi$ is an arbitrary element of 
$\cald(0,T)$, the $L^2(0,T)$
function in the brackets is necessarily equal to zero.
We conclude observing that 
$a( \Psi(t), \varphi) = \la A_V^* \varphi, \Psi(t) \ra_V$
for a.a. $t$. 
\end{proof}

\begin{theorem}
Let hypothesis \eqref{hyp-Goh-tr1} hold. Then
the compactness condition 
\eqref{compachyp} is satisfied, and 
problem \eqref{problem:OC}
 has a nonempty set of minima.
\end{theorem}

\begin{proof}
By our hypotheses, the mapping 
$f\mapsto\Psih$ is continuous from
$L^2(0,T)$ into $W(0,T)$.
By \eqref{Aubin}, 
the mapping $u\mapsto\Psih[u]$ is compact
from $L^2(0,T)$ to $L^2(0,T;\H)$.
So, the compactness hypothesis
\eqref{compachyp} holds, and the existence of
a minimum follows from Theorem \ref{existsol}.
\end{proof} 

\subsection{Heat equation}

\subsubsection{Statement of the problem}
We first write the optimal control in an informal way.
Let $\Om$ be a bounded open subset
of $\RR^n$ with $C^2$ boundary. 

The state equation, where $y=y(t,x)$,  is
 \be
\label{eq:heat-1}
\left\{
\begin{aligned}
  \frac{\partial y(t,x)}{\partial t} +\A_\H y(t,x)&=f(t,x)+u(b_1(x)+b_2(x) y(t,x)) &&
  \text{in }(0,T)\times \Om, \\
  y(0,x)&=y_0(x)  && \text{in } \Om, \\
  y(t,x)&=0 && \text{on } (0,T) \times \partial\Om.
\end{aligned}
\right.
\ee
Here $\A_\H$ stands for the differential operator
in divergence form, for $(t,x)\in (0,T)\times \Om$:
\be
\label{A_elliptic}
(\A_\H y)(t,x)= -\sum_{j,k=1}^n \frac{\partial}{\partial x_k} 
\left[a_{jk}(x) \frac{\partial y(t,x)}{\partial x_j}  \right],
\ee
where  $a_{jk}\in C^{0,1}(\bar \Om)$ satisfy, for each $x\in \bar\Om,$
the symmetry hypothesis 
$a_{jk} = a_{kj}$
as well as, for some $\nu>0$:
\be
\sum_{j,k=1}^n a_{jk}(x) \xi_j \xi_k \ge \nu |\xi|^2,\quad 
\text{for all $\xi\in \RR^n$, $x\in\Om$.}
\ee
Let $H :=L^2(\Om)$ and $V = H^1_0(\Om)$. 
We apply the abstract framework with 
$\H$ equal to $H$.
We choose $\dom(\A_\H) := H^2(\Om)\cap V$.
The pair $(H,V)$ satisfies the hypothesis of the abstract 
parabolic setting, namely, that $V$ is
continuously embedded in $H$, with dense and compact inclusion.
We next define $A_V\in L(V,V^*)$ by 
\be
\la A_V y,z \ra_V := \sum_{j,k=1}^n \int_\Om 
a_{jk}(x)  \frac{\partial y }{\partial x_j } 
 \frac{\partial z }{\partial x_k } \dd x,
\quad \text{for all $y$, $z$ in $V$}.
\ee
The bilinear form over $V$
defined by 
$a(y,z):= \la A_V y,z \ra_V$ 
is continuous and 
satisfies the semicoercivity condition \eqref{semicoercive}.
Since $A_V y=\A_\H y$ for all $y$ in $H^2(\Om)\cap V$,
$A_\H$ is nothing but the generator of the 
semigroup built in the previous section.
This semigroup is contracting, since 
the Hille Yosida characterization of a generator 
given in Lemma \ref{lem:general_parab_setting}
holds with $M=1$, $n=1$ and $\omega=0$.

In the sequel of this study of the heat equation,
we assume 
\be
\label{heat:data}
y_0 \in H; \quad
f\in C(0,T;H),
\quad
 b_1\in \dom(\A_\H), 
\quad 
b_2\in W^{2,\infty}_0(\Om).
\ee
The corresponding data of the abstract theory are
$\B_1:=b_1$ and $\B_2\in \L(H)$ defined by 
$(\B_2 y)(x) := b_2 (x) y(x)$ for $y$ in  $\calh$ and
$x\in \Om$.
By Lemma  \ref{lem:general_parab_setting}, equation
\eqref{eq:heat-1} has a mild solution $y$ in $C(0,T;\calh)$
which coincides with the variational solution in the sense of
\eqref{eq-vf1}.

The cost function is, given $\alpha \in \cR$: 
\be
\label{def-cost-fun-heat}
\ba{lll}
J(u, y) := & \disp 
\alpha \int_0^T  u(t) \dd t 
+
\half\int_{(0,T)\times\Om}
( y(t,x)-y_d(t,x) )^2 \dd x \dd t 
\\ & \disp \hspace{5mm}
+
 \half \int_\Om (y(T,x)-y_{dT}(x))^2 \dd x.
\ea\ee

We assume that 
\be
\label{hyp-yd-heat}
\quad 
y_d \in C(0,T;H); \quad  y_{dT} \in V.
\ee
For $u\in L^1(0,T)$, write the reduced cost as 
$F(u) := J(u,y[u])$.
The optimal control problem is, $\U_{ad}$
being defined in \eqref{Uad}:
\be
\label{pb-parab}
\Min F(u); \quad u \in \U_{ad}.
\ee

\subsubsection{Commutators}

Given $y\in\dom(\A_\H)$, we have by \eqref{A_elliptic}
that
\be
\label{rem-bracket-heat1a}
\ba{lll} 
M_1 y &=&
(\A_\H \B_2 - \B_2 \A_\H) y 
\\&=& \disp
-\sum_{j,k=1}^n \left( 
\frac{\partial}{\partial x_k} 
\left[a_{jk} \frac{\partial}{\partial x_j}  (b_2 y) \right]
-b_2 \frac{\partial}{\partial x_k} 
\left[a_{jk} \frac{\partial y}{\partial x_j}  \right] 
\right)
\\ &=& \disp
-\sum_{j,k=1}^n \left( 
\frac{\partial}{\partial x_k} 
\left[ b_2 
( a_{jk} \frac{\partial y}{\partial x_j} )  
+ a_{jk} 
y \frac{\partial b_2}{\partial x_j}  
\right]
-b_2 \frac{\partial}{\partial x_k} 
\left[a_{jk} \frac{\partial y}{\partial x_j}   \right] 
\right)
\\ &=& \disp
-\sum_{j,k=1}^n \left( 
\frac{\partial b_2}{\partial x_k}  
\left[a_{jk} 
\frac{\partial y}{\partial x_j}  \right]
+ \frac{\partial }{\partial x_k}   \left[
a_{jk} 
y \frac{\partial b_2}{\partial x_j}  
\right]
\right).
\ea\ee
As expected, this commutator is
 a first order differential operator that has a continuous extension
to the space $V$. 
In a similar way we can check that $[M_1,\B_2]$
is the ``zero order'' operator given by 
\be
[M_1,\B_2] y = - 2 \sum_{j,k=1}^n a_{jk} 
\frac{\partial b_2}{\partial x_j}  
\frac{\partial b_2}{\partial x_k}  y.
\ee

\begin{remark}
\label{rem-bracket-heat-pc}
In the case of the Laplace operator, i.e. when
$a_{jk} = \delta_{jk}$,  we find that
\be
\label{rem-bracket-heat1b}
M_1 y = (\A_\H \B_2 - \B_2 \A_\H) y 
=
2 \nabla b_2 \cdot \nabla y + y \Delta b_2;
\quad
[M_1,\B_2] y = 2 y | \nabla b_2 |^2,
\ee 
and then for $p\in V$: 
\be
\ba{lll}
( M^*_1 p,y)_\H &= \disp \int_\Om 
\left ( 2 \nabla b_2 \cdot \nabla y + y \Delta b_2 \right) p \dd x
\vspace{1mm}\\ & =\disp 
\int_\Om 
\left ( -2 \ddiv ( p \nabla b_2 )+ p \Delta b_2\right) y \dd x
\vspace{1mm}\\ & =\disp
\int_\Om 
\left ( 2 \nabla p \cdot \nabla b_2 - p \Delta b_2\right) y \dd x
\ea\ee
so that  we can write
\be
M^*_1 p = 2 \nabla p \cdot \nabla b_2 - p \Delta b_2.
\ee
We have similar expressions for 
$M_2$ and $M^*_2$, replacing $b_2$ by 
$b^2_2$. 
\end{remark}

\subsubsection{Analysis of the optimality conditions}
For the sake of simplicity we only discuss the case of the 
Laplace operator and assume that 
$b_1(x)=0$ for all $x\in \Om$. 
The costate equation is then 
\be
- \dot p - \Delta p = y - y_d + u b_2 p \; 
\text{ in $(0,T)\times \Om$;}  \quad  p(T) = y(T) - y_{dT}. 
\ee
Recalling the expression of $b^1_{z}$
in \eqref{equ-bunz}, we obtain that 
the equation for $\xi := \xi_z$ introduced in \eqref{diffeqxi} 
reduces to 
\be
\label{xi-equ-heat}
\dot\xi - \Delta \xi  = \uh b_2 \xi - w 
(b_2 f  + 2\nabla b_2\cdot \nabla y - y \Delta b_2)
\; \text{ in $(0,T)\times \Om$;}  \quad  \xi(0) =0. 
\ee
The quadratic forms $\Q$ and $\hatQ$ defined in \eqref{tildeQ} and \eqref{Omega}
 are as follows:
\be 
  \Q(z,v) = \int_0^T \Big( \norm{z(t)}{H}^2
 + 2 v(t) ( \ph(t),b_2  z(t) )_H  \Big)\dd t + \norm{z(T)}{H}^2,
 \ee
and as we recall from our general framework
\be
\hatQ(\xi,w,h) = \hatQ_T(\xi,h) +\hatQ_a(\xi,w)+\hatQ_b(w),
\ee
with $\hatQ_b(w)= \int_0^T w^2(t) R(t)\dd t$,
$R \in C(0,T)$, and 
\begin{align}\label{OmegaT-a}
\hatQ_T(\xi,h)&= \norm{\xi(T) + h b_2\yh(T)}{H}^2
 + h^2 ( \ph(T),b_2^2 \yh(T))_H +h ( \ph(T),b_2 \xi(T))_H,
\\
\hatQ_a(\xi,w)&= \int_0^T \Big( \norm{\xi}{H}^2 + 2 w(
2 b_2 \yh  -   b_2 y_d -
 2 \nabla \ph \cdot \nabla b_2 + \ph \Delta b_2, \xiz )_H   
\Big) \dd t,
\\
R(t)&=  \norm{b_2\yh}{H}^2 + (\yh-y_d,b_2^2\yh)_H + (\ph(t),b_2^2 f(t) 
- 
2|\nabla b_2|^2  \yh)_H.
\end{align}

\begin{theorem} 
Let $\uh$ be a weak minimum for problem \eqref{pb-parab}.
Then 
{\rm (i)}
the second order necessary condition 
\eqref{lem3-ome} holds, i.e.,
\be
\label{lem3-ome-haet}
  \hatQ (\xi[w],w,h) \ge 0\quad \text{for all } (w,h) \in PC_2(\uh),
  \ee
 {\rm (ii)} 
$R(t)\geq 0$ over singular arcs, 
\\ {\rm (iii)} 
if additionally \eqref{finite_structures}-\eqref{RposTBB} 
are satisfied, then the second order optimality condition
\eqref{sufcondso} holds iff  the quadratic growth condition
\eqref{sufcondqg}  is satisfied.
\end{theorem} 

\begin{proof}
 (i)
It suffices to check the hypotheses for Lemma \ref{lem3}.
Relations  \eqref{hyp-Goh-tr1}, where we choose 
$E_1 := V$, follows from 
\eqref{heat:data}, \eqref{hyp-yd-heat},
and the above computation of commutators.
Since $y_{dT} \in V$ we have that 
\be
\ph \in L^2(0,T;V\cap H^2(\Om))\cap H^1(0,T;H)\subset C(0,T;V),
\ee
so that $M^*_1 \ph \in C(0,T;H)$.
Point (i) follows. 
\\ (ii)
This follows from Corollary \ref{cor2},
the compactness hypothesis \eqref{comp-hyp-xi}
being a standard result.
\\ (iii)
We apply Theorem \ref{thm-sufcond},
which assumes hypothesis \eqref{hyp-Goh-tr2}, and the latter 
are satisfied in our present setting. 
\end{proof}


\begin{remark}
In the present framework, the generator of the semigroup is
diagonalizable with a sequence of real eigenvalues
$\mu_k \rar \infty$. 
By \eqref{sec:notation-sch-d-h2}, the space $\H^2$ 
of section \ref{sec:notation-sch-d} coincides with 
$H^2\cap V$. 
\end{remark}

\begin{remark}
It is not difficult to extend such results
 for more general differential operators of the type,
where the $a_{jk}$ are as before, 
$b\in C^{0,1}(\Om)^n$ and $c\in C^{0,1}(\Om)^n$:
\be \label{general_diff_op}
(\A_\H y)(t,x)= -\sum_{j,k=1}^n \frac{\partial}{\partial x_k} 
\left[a_{jk}(x) \frac{\partial}{\partial x_j}  y (t,x) \right]
 + \sum_{j=1}^n \frac{\partial ( b_{j}(x) y(t,x)) }{\partial x_j}    + c y(t,x).
\ee
\end{remark}

\subsection{Wave equation}
\subsubsection{Statement of the problem} 
Again, let $\Omega$ be an 
open bounded subset of $\RR^n$ with $C^2$
boundary.
 The state equation is 
\be
\label{eq:wave}
\left\{
\begin{aligned}
  \frac{\partial^2 y_1(t,x)}{\partial t^2} 
 + \A_\H y_1(t,x) &=f_2(t,x)+u(b_1(x) + b_2(x) y_1(t,x))&& \text{in
 }(0,T)\times \Om, 
\\
  y_1(0,x)&=y_{01}(x), \quad 
\frac{\partial}{\partial t}y_1(0,x)=y_{02}(x) && \text{in } \Om, \\
  y_1(t,x)&=0 && \text{on } (0,T) \times \partial\Om,
\end{aligned}
\right.
\ee
with $\A_\H$ as defined in \eqref{A_elliptic},
and again $a_{jk}\in C^{0,1}(\bar \Om)$.
Setting $y_2(t) := \dot y_1(t)$,
we can reformulate the state equation as a 
first-order system in time given by
 \be
\label{eq:wave-semigroup}
\begin{aligned}
  \dot y + \A_W y &=f+u(\B_1 + \B_2 y) \quad t\in (0,T), \quad y(0)=y_0, 
\end{aligned}
\ee
with 
\be\label{wave:matrices}
\A_W := \begin{pmatrix} 0 & -\id \\ \A_\H & 0\end{pmatrix},\quad 
\B_1:=\begin{pmatrix}
      0 \\ b_1
     \end{pmatrix},
\quad
\B_2:=\begin{pmatrix}
      0 & 0\\ b_2 &0 \\
     \end{pmatrix},
     \quad
     f:=\begin{pmatrix}
           0 \\ f_2
          \end{pmatrix},\quad y_0=\begin{pmatrix} y_{01} \\ y_{02}
          \end{pmatrix}.
\ee
Set $\calh:=V \times H$ with $V:=H^1_0(\Om)$ and $H:=L^2(\Om)$.
We 
endow the space $\calh$ with the norm 
\be
\| y\|_\H := \left( \|y_1\|^2_V + \|y_2\|^2_H,
\right)^{1/2},
\ee
where for $z\in V$:
\be
\|z\|_V := \left( \sum_{i,k=1}^n  a_{i,k}(x) \int_\Om
\frac{\partial z}{\partial x_j}\frac{\partial z}{\partial x_k}
\dd x \right)^{1/2}.
\ee
It is known that $\A_W$ is the generator of a contraction semigroup with
$\dom(\A_W)=H^2_{0,1} \subset \calh$, where 
\be 
H^2_{0,1}:=( H^2(\Om)\cap  V)  \times  V.
\ee 
We verify the Hille Yosida characterization of a generator of a
contraction semigroup \eqref{Hille_Yosida} with $n=1$ and $\omega=0$
given by,
for $\lambda>0$:
\be\label{energy_est}
\norm{y_1}{V}^2 + \norm{y_2}{H}^2 
\le 
\frac{1}{\lambda ^2} \left(\norm{f_1}{V}^2 + \norm{f_2}{H}^2\right).
\ee
Indeed, consider for $(f_1,f_2)\in  \H$
the system 
\begin{align}
\lambda (y_1,v)_{ V } - (y_2,v)_{ V } &=(f_1,v)_{ V }, && \text{for all }v\in  V ,\\
(A_\H y_1,w)_{H} + \lambda (y_2,w)_{H} &= (f_2,w)_{H}, && 
\text{for all }w\in H.
\end{align}
Estimate \eqref{energy_est} follows by setting $(v,w)=(y_1,y_2)$, 
 adding the two equations, and using the Cauchy Schwarz
inequality.
Taking $\lambda=0$ and $f=0$ we obtain by similar
arguments that the operator $\A_W$ is antisymmetric. 
One can also rely on the eigenvector decomposition.
See more in  
\cite[p. 59, vol. I]{MR2273323}. 

In this section we assume
\begin{align}\label{wave1}
 (y_{01},y_{02}) \in H^{2,1}_0,\quad
 b_1    & \in H^2(\Om)\cap V, \;\;  b_2     \in
 W^{1,\infty}_0(\Om),\quad  f_2 \in L^2(0,T;V).
\end{align}

 \begin{lemma}\label{lem:wave1}
Under the assumptions \eqref{wave1} equation
\eqref{eq:wave} has a unique mild
solution $y$ in $C(0,T;H^{2,1}_0)$.
\end{lemma}

\begin{proof}
Consequence of Remark \ref{lem-reg.l1-r}.
\end{proof}

Furthermore, let the cost be given by 
\be
\label{def-cost-fun-wawe}
\ba{lll}
J(u, y) := \disp
\alpha \int_0^T  u(t) \dd t 
+
\half\int_0^T
\| y(t)-y_d(t) \|^2_\H \dd t 
+
 \half \| y(T)-y_{dT} \|^2_\H.
\ea
\ee
We assume that 
\be\label{wave:yd}
y_d \in C(0,T;\calh); \quad y_{dT} \in \calh.
\ee
For $u\in L^1(0,T)$, write the reduced cost as 
$F(u) := J(u,y[u])$.
The optimal control problem is, $\U_{ad}$
being defined in \eqref{Uad}:
\be
\label{pb-cont-wave}
\Min F(u); \quad u \in \U_{ad}.
\ee 


\begin{lemma}
\label{wave-exist1}
Problem \eqref{pb-cont-wave} has at least one minimum.
\end{lemma}

\begin{proof}
Set $\tilde H:=L^2(0,T;\H)$.
By Aubin's Lemma, the mapping 
$f\mapsto b_2 y_1[y_0,f]$
is compact from $\tilde H$
into $L^2(0,T;H)$. 
Indeed, it is continuous
from $\tilde H$ to 
$L^2(0,T; V )
\cap H^1(0,T; H)$.
We then easily pass to the limit in a minimizing sequence
in the nonlinear term of the state equation, that involves
only the first component of the state.
\end{proof}

\subsubsection{Commutators} We have
\be 
M_1 =\begin{pmatrix}- b_2 & 0 \\ 0 &b_2\end{pmatrix};
\quad
[M_1,\B_2] =\begin{pmatrix} 0 & 0 \\ 2 b^2_2 &0\end{pmatrix}; \quad M_2=0.
\ee
Here the commutator $M_1$ is a zero order differential operator.
\subsubsection{Analysis of optimality conditions}
Again, for the sake of simplicity we only discuss the case of the 
Laplace operator and assume that 
$b_1(x)=0$ for all $x\in \Om$. 

\begin{lemma} 
Let $d_1\in W^{1,\infty}(\Om)$ and 
$d_2$, $d_3$ belong to $L^\infty(\Om)$.
Define $N\colon \H \rar \H$ by 
\be
N y := \begin{pmatrix} d_1 & 0 \\ d_2 & d_3\end{pmatrix}
\begin{pmatrix} y_1 \\ y_2\end{pmatrix}
=
\begin{pmatrix} d_1 y_1 \\ d_2 y_1 + d_3 y_2\end{pmatrix}.
\ee
Then with the same convention 
\be
N^*v =
\begin{pmatrix}  
A_V^{-1}(d_1 A_V v_1 ) +   A_V^{-1}(d_2 v_2) \\ d_3 v_2
\end{pmatrix}.
\ee
\end{lemma} 

\begin{proof} 
Let $y$, $z$ belong to $\H$, then 
\be
(z,Ny)_\H =  (z_1,d_1 y_1)_V+(z_2,d_2 y_1)_H +(z_2,d_3 y_2)_H.
\ee
Clearly
\be
(z_1,d_1 y_1)_V = \la A_V z_1, d_1 y_1\ra_V  
= (A_V^{-1}(d_1 A_V z_1), y_1)_V.
\ee
Now
\be
(z_2,d_2 y_1)_H =
(d_2 z_2, y_1)_H = (A_V^{-1}(d_2 z_2), y_1)_V.
\ee
Finally 
\be
(z_2,d_3 y_2)_H = (d_3 z_2, y_2)_H. 
\ee
The result follows.
\end{proof} 
Note that the above results uses the fact that $A_V$
is a symmetric operator. 
As a consequence
\be
M^*_1 \ph = \begin{pmatrix}  
- A_V^{-1}(b_2 A_V \ph_1)\\ b_2 \ph_2
\end{pmatrix}; \\
\quad
( M^*_1 \ph, \xi)_\H = 
- \la b_2 A_V \ph_1, \xi_1\ra_V + (b_2 \ph_2,\xi_2)_H.
\ee
One easily checks that 
$\A^*_W 
= \begin{pmatrix} 0 & \id \\ -\A_\H & 0\end{pmatrix}$
has the same domain as $\A_W$.
Therefore the costate equation reads 
\be
\left\{ \ba{lll}
- \dot p_1 - p_2 &= u \A_\H^{-1}(b_2 \ph_2) + y_1 - y_{1d},
\\
- \dot p_2  + \A_\H p_1 &= y_2 - y_{2d},
\ea\right. \ee
with final condition 
$p(T) = y_{dT}$.

 The equation in $\xi:=\xi_z$ introduced in \eqref{diffeqxi} is given by
\be\label{xi-wave}
\dot\xi + \A_W \xi  = \uh \B_2 \xi + w b^1_z; \quad \xi(0)=0
\quad\text{with}\quad
b^1_{z} 
=
-\B_2 f -M_1 \yh.
\ee
Since $\B_2f=0$ the dynamics for $\xi$ reduces to 
\be
\label{xi-wave-comp} \left\{\ba{lll}
\dot\xi_1 - \xi_2 &= w b_2 \yh_1,
\\
\dot\xi_2 + \A_\H \xi_1 &= \uh b_2 \xi_1 - w b_2 \yh_2.
\ea\right.\ee
The quadratic forms $\Q$ and
$\hatQ$ defined in \eqref{tildeQ} and  \eqref{Omega}:
First
\be
  \Q(z,v) = \int_0^T \Big( \norm{z(t)}{\H}^2
 + 2 v(t) ( \ph_2(t), b_2  z_1(t) )_H  \Big)\dd t + \norm{z(T)}{\H}^2,
 \ee
and second, $ \hatQ(\xi,w,h) = \hatQ_T(\xi,h) +\hatQ_a(\xi,w)+\hatQ_b(w)$,
where 
\be
\hatQ_b(w)= \int_0^T w^2(t) R(t)\dd t.
\ee
 Here, $R\in C(0,T)$ and
\begin{align}\label{OmegaT-w}
\hatQ_T(\xi,h)&= \norm{\xi_1(T)}{V}^2  + \norm{\xi_2(T)+h b_2 \yh_1(T)}{H}^2
  +h ( \ph_2(T),b_2 \xi_1(T) )_H,
\\
\hatQ_a(\xi,w)&= \int_0^T \Big( \norm{\xi}{\H}^2 + 2 w (\xi_2,b_2 \yh_1)_H \Big)\dd t \\
 &\quad + \int_0^T \Big(2 w (\yh_2 - y_{2d},b_2\xiz_1)_H 
+  2w ( \la b_2 A_V \ph_1, \xi_1\ra_V - (b_2 \ph_2,\xi_2)_H)
\Big) \dd t, \\
R(t)&=  \norm{b_2\yh_2}{H}^2  - 2(\ph_2(t), b_2^2 \yh_1)_H.
\end{align}

\begin{theorem} 
Let $\uh$ be a weak minimum for problem \eqref{pb-cont-wave}.
Then
{\rm (i)}
The second order necessary condition 
\eqref{lem3-ome} holds, i.e.,
\be
  \hatQ (\xi[w],w,h) \ge 0\quad \text{for all } (w,h) \in PC_2(\uh).
  \ee
 {\rm (ii)} 
$R(t)\geq 0$ over singular arcs. 
\\ {\rm (iii)} 
Let \eqref{finite_structures}-\eqref{RposTBB} 
hold. Then the second order optimality condition
\eqref{sufcondso} holds iff  the quadratic growth condition
\eqref{sufcondqg}  is satisfied.
\end{theorem} 

\begin{proof}
(i) Again, it suffices to check the hypotheses for Lemma \ref{lem3}.
Relations  \eqref{hyp-Goh-tr1}, where we choose 
$E_1 := \H$, follow from \eqref{wave:matrices}, and \eqref{wave:yd}, and the above computation of the commutator which contains no derivative.
In particular $M^*_1 \ph \in C(0,T;\H)$.
Point (i) follows. 

(ii) To apply Corollary \ref{cor2} we check the compactness hypothesis \eqref{comp-hyp-xi}.
We have
\be
w\mapsto \xi[w],\quad L^2(0,T) \rightarrow L^2(0,T;\H),
\ee
with $\xi[w]$ being the solution of \eqref{xi-wave-comp}. Since
$
 \xi[w]\in Z:=C(0,T; H^2_{0,1}(\Om))$
and 
$\dot \xi[w] \in L^2(0,T;H\times H^{-1}(\Om))$.
Since $H^2_{0,1}$ is compactly embedded in $\H$, 
and $\H \subset H\times H^{-1}(\Om)$
with continuous inclusion,
we conclude by Aubin's Lemma.
\\ (iii) We apply Theorem \ref{thm-sufcond},
which assumes hypothesis \eqref{hyp-Goh-tr2}, 
which is satisfied in our present setting. 
\end{proof}

\begin{remark}
 As for the heat equation the framework can be extended to more general differential operators $\A_\H$ of type \eqref{general_diff_op}.
\end{remark}

\bibliographystyle{amsplain}
\bibliography{lit,hjb}
\end{document}